\documentclass[10pt, a4paper]{amsart}

\usepackage{amssymb,amsfonts}
\usepackage[all,arc]{xy}
\usepackage{enumerate}
\usepackage{mathrsfs}
\usepackage{tikz-cd}
\usetikzlibrary{decorations.pathmorphing}
\usepackage{enumitem}
\usepackage{empheq}
\usepackage{kotex}
\usepackage{ifpdf}
\usepackage{mathtools}
\usepackage{pst-node, pst-plot, pstricks}
\usepackage{mathrsfs}
\usepackage{amsmath}
\usepackage{epsfig}
\usepackage{amscd}
\usepackage{graphicx, color}
\graphicspath{ {./images/} }

\def\E{\ifmmode{\mathbb E}\else{$\mathbb E$}\fi} 
\def\N{\ifmmode{\mathbb N}\else{$\mathbb N$}\fi} 
\def\R{\ifmmode{\mathbb R}\else{$\mathbb R$}\fi} 
\def\Q{\ifmmode{\mathbb Q}\else{$\mathbb Q$}\fi} 
\def\C{\ifmmode{\mathbb C}\else{$\mathbb C$}\fi} 
\def\H{\ifmmode{\mathbb H}\else{$\mathbb H$}\fi} 
\def\Z{\ifmmode{\mathbb Z}\else{$\mathbb Z$}\fi} 
\def\P{\ifmmode{\mathbb P}\else{$\mathbb P$}\fi} 
\def\T{\ifmmode{\mathbb T}\else{$\mathbb T$}\fi} 
\def\SS{\ifmmode{\mathbb S}\else{$\mathbb S$}\fi} 
\def\DD{\ifmmode{\mathbb D}\else{$\mathbb D$}\fi} 

\DeclareSymbolFont{yhlargesymbols}{OMX}{yhex}{m}{n}

\DeclareMathAccent{\widetriangle}{\mathord}{yhlargesymbols}{"E6}

\def\E{\ifmmode{\mathbb E}\else{$\mathbb E$}\fi} 
\def\N{\ifmmode{\mathbb N}\else{$\mathbb N$}\fi} 
\def\R{\ifmmode{\mathbb R}\else{$\mathbb R$}\fi} 
\def\Q{\ifmmode{\mathbb Q}\else{$\mathbb Q$}\fi} 
\def\C{\ifmmode{\mathbb C}\else{$\mathbb C$}\fi} 
\def\H{\ifmmode{\mathbb H}\else{$\mathbb H$}\fi} 
\def\Z{\ifmmode{\mathbb Z}\else{$\mathbb Z$}\fi} 
\def\P{\ifmmode{\mathbb P}\else{$\mathbb P$}\fi} 
\def\T{\ifmmode{\mathbb T}\else{$\mathbb T$}\fi} 
\def\SS{\ifmmode{\mathbb S}\else{$\mathbb S$}\fi} 
\def\DD{\ifmmode{\mathbb D}\else{$\mathbb D$}\fi} 

\newcommand{\ben}{\begin{enumerate}}
\newcommand{\een}{\end{enumerate}}
\newcommand{\be}{\begin{equation}}
\newcommand{\ee}{\end{equation}}
\newcommand{\bea}{\begin{eqnarray}}
\newcommand{\eea}{\end{eqnarray}}
\newcommand{\beastar}{\begin{eqnarray*}}
\newcommand{\eeastar}{\end{eqnarray*}}
\newcommand{\bc}{\begin{center}}
\newcommand{\ec}{\end{center}}

\theoremstyle{theorem}
\newtheorem{thm}{Theorem}[section]
\newtheorem{cor}[thm]{Corollary}
\newtheorem{lem}[thm]{Lemma}
\newtheorem{prop}[thm]{Proposition}
\newtheorem{'thm'}[thm]{'Theorem'}

\theoremstyle{definition}
\newtheorem{defn}[thm]{Definition}
\newtheorem{rem}[thm]{Remark}

\newtheorem{exam}[thm]{Example}

\newtheorem{proof-sketch}[thm]{Proof-Sketch}

\newtheorem{claim}[thm]{Claim}

\newtheorem{lem-defn}[thm]{Lemma-Definition}
\newtheorem{prop-defn}[thm]{Proposition-Definition}
\newtheorem{thm-defn}[thm]{Theorem-Definition}

\newtheorem*{thm*}{Theorem}

\numberwithin{equation}{section}

\def\R{{\mathbb R}}

\def\E{{\mathbb E}}
\def\Z{{\mathbb Z}}
\def\C{{\mathbb C}}
\def\R{{\mathbb R}}
\def\P{{\mathbb P}}

\def\N{{\mathbb N}}

\def\11{{\mathbb I}}

\def\sgn{{\text{\rm sgn}}}

\def\C{\mathbb{C}}
\def\Z{\mathbb{Z}}

\def\T{\mathbb{T}}

\def\Q{\mathbb{Q}}

\def\E{\ifmmode{\mathbb E}\else{$\mathbb E$}\fi} 
\def\N{\ifmmode{\mathbb N}\else{$\mathbb N$}\fi} 
\def\R{\ifmmode{\mathbb R}\else{$\mathbb R$}\fi} 
\def\Q{\ifmmode{\mathbb Q}\else{$\mathbb Q$}\fi} 
\def\C{\ifmmode{\mathbb C}\else{$\mathbb C$}\fi} 
\def\H{\ifmmode{\mathbb H}\else{$\mathbb H$}\fi} 
\def\Z{\ifmmode{\mathbb Z}\else{$\mathbb Z$}\fi} 
\def\P{\ifmmode{\mathbb P}\else{$\mathbb P$}\fi} 
\def\SS{\ifmmode{\mathbb S}\else{$\mathbb S$}\fi} 
\def\DD{\ifmmode{\mathbb D}\else{$\mathbb D$}\fi} 

\def\R{{\mathbb R}}

\def\E{{\mathbb E}}
\def\Z{{\mathbb Z}}
\def\C{{\mathbb C}}
\def\R{{\mathbb R}}

\def\N{{\mathbb N}}

\def\darr#1{\raise1.5ex\hbox{$\leftrightarrow$}
\mkern-16.5mu #1}

\def\roughly#1{\raise.3ex\hbox{$#1$\kern-.75em
\lower1ex\hbox{$\sim$}}}

\def\opname#1{\mathop{\kern0pt{\rm #1}}\nolimits}

\def\Incl{\operatorname{Incl}}
\def\Eval{\operatorname{Eval}}
\def\Err{\operatorname{Err}}

\begin{document}

\quad \vskip1.375truein

\bibliographystyle{plain}

\title{Homotopy models for $L_{\infty}[1]$-algebras in higher degrees}

\author{Taesu Kim}


\begin{abstract}
We propose a model of higher homotopy theory of $L_{\infty}[1]$-morphisms as a natural generalization of the $A_{\infty}$-homotopies defined by Fukaya-Oh-Ohta-Ono \cite{FOOO1}. Within this framework, we show that a filling condition holds for simplices whose vertices are assigned quasi-isomorphisms.
\end{abstract}

\keywords{$L_{\infty}[1]$-algebras, Homotopies, Higher homotopies, Quasi-isomorphism}
\subjclass[2020]{Primary 17B55; Secondary 55U35, 18M70}
\thanks{This work was carried out while the author was affiliated with the Department of Mathematics at Pohang University of Science and Technology (POSTECH) and was supported by the BK21 FOUR program funded by the Ministry of Education (MOE), Korea (No. 4120240414885)}

\maketitle

\date{}

\tableofcontents

\section{Introduction}

Fukaya-Oh-Ohta-Ono \cite{FOOO1} constructed a filtered $A_{\infty}$-algebra structure on the cohomology models of Lagrangian submanifolds in a symplectic manifold. They also defined a model of homotopies between $A_{\infty}$-algebra morphisms to address the geometric implications arising from families of almost complex structures on the symplectic manifold. Their construction begins with the following definition:

\begin{defn}\cite[Definition 4.2.1]{FOOO1}\label{ahgadsfjkl}
Let $C$ be a filtered $A_{\infty}$-algebra. We say a filtered $A_{\infty}$-algebra $\mathfrak{C}$ is a \textit{model} of $\Delta^1 \times C$ if there exist a filtered $A_{\infty}$-morphisms
\[
\Eval_j : \mathfrak{C} \rightarrow C, \ \ j = 0,1\\
\]
and
\[
\Incl : {C} \rightarrow \mathfrak{C},
\]
with the following properties:
\begin{enumerate}[label = (\roman*)]
\item $\left(\Eval_j\right)_{k \geq 2} \equiv 0, \ j =0,1, \ \Incl_{k \geq 2} \equiv 0.$
\item $\Eval_j, \ j =0,1$ and  $\Incl$ are quasi-isomorphisms.
\item $(\Eval_j)_1 \circ \Incl_1 = \mathrm{id}_C.$
\item $(\Eval_0)_1 \oplus (\Eval_{1})_1 : \mathfrak{C} \rightarrow {C} \oplus C$ is surjective.
\end{enumerate}
\end{defn}
In the original definition \cite[Definition 4.2.1]{FOOO1}, condition (ii) requires that $(\Eval_j)_1$ and $\Incl_1$​ be chain homotopy equivalences. In our setting, however, we always work over a field, where chain homotopy equivalence coincides with quasi-isomorphism.

From Definition \ref{ahgadsfjkl}, homotopies between two $A_{\infty}$-algebras are defined as follows.

\begin{defn}\cite[Definition 4.2.35]{FOOO1}
We say that two filtered $A_{\infty}$-morphisms $f_0, f_1 : (C, \{l_k\}) \rightarrow (C', \{l'_k\})$ are \textit{homotopic} if there exist a \textit{model of} $\Delta^1 \times C'$ denoted by $\mathfrak{C}'$ and a filtered $A_{\infty}$-morphism $h: C \rightarrow \mathfrak{C}'$ such that we have $f_j = \Eval_j \circ h, \ j = 0,1.$
\end{defn}

We remark that \cite[Subsection 3.7]{Keller} provides another definition of homotopy between $A_{\infty}$-morphisms.

In fact, their constructions apply to other types of strong homotopy algebras as well. In this paper, we develop an $L_{\infty}[1]$-algebra version of their theory and propose a model for its generalization to higher degrees, naturally extending the definitions in \cite{FOOO1}.

We make the following definition for the $n \geq 2$ cases (see Definition \ref{hhtp}):

\begin{defn}[Models of $\Delta^n \times C$]
Let $C$ be an $L_{\infty}[1]$-algebra. Suppose that we have defined models of $\Delta^k \times C$ with $k \leq n-1.$ We recursively define \textit{models of} $\Delta^n \times C$ with $n \geq 2$ to be a collection of $L_{\infty}[1]$-algebras
\begin{equation}\nonumber
\mathfrak{C}^{(n)}, \  \left(\mathfrak{C}^{(n)}\right)_{J} \equiv \mathfrak{C}_{J}^{(n-1)}, \text{ where } J \text{ is a subset of } \{0, \cdots, n\} \text{ such that } |J| = n,
\end{equation}
together with an $L_{\infty}[1]$-morphisms
\[
\Eval_J^{(n)} :  \mathfrak{C}^{(n)} \rightarrow \mathfrak{C}_{J}^{(n-1)},
\]
and
\[
\Incl^{(n)} : {C} \rightarrow \mathfrak{C}^{(n)}\\
\]
with the following properties:
\begin{enumerate}[label = (\roman*)]
\item $\mathfrak{C}_{J}^{(n-1)}$ is a model of $\Delta^{n-1} \times C$ with $\mathfrak{C}_{\{i\}}^{(0)} = C$ for each $i.$
\item $\left(\mathfrak{C}_{J}^{(n-1)} \right)_{J'} = \left(\mathfrak{C}_{J'}^{(n-1)} \right)_{J} = \mathfrak{C}_{J \cap J'}^{(n-2)}$ for all $J, J' \subset \{0, \cdots, n\}$ with $|J| = |J'| = n$ and $|J \cap J'| = n-1.$ 
\item $\left(\Eval_j\right)_{k \geq 2} \equiv 0, \ j =0,1, \ \Incl_{k \geq 2} \equiv 0.$
\item $\Eval_J^{(n)}$ and $\Incl^{(n)}$ are quasi-isomorphisms.
\item $\Eval_J^{(n)} \circ \Incl^{(n)} = \Incl_J^{(n-1)},$
where $ \Incl_J^{(n-1)}$ is the Incl map for $\mathfrak{C}_{J}^{(n-1)},$ the model of $\Delta^{n-1} \times C$ for the index $J.$
\item The following sequence of chain complexes
\[
\begin{split}
\mathfrak{C}^{(n)} \xrightarrow{\partial_n}  \bigoplus_{\substack{J \subset \{0, \cdots, n\}, \\ |J| = n}} &\mathfrak{C}_J^{(n-1)} \xrightarrow{\partial_{n-1}} \bigoplus_{\substack{J' \subset \{0, \cdots, n\}, \\ |J'|=n-1}} \mathfrak{C}_{J'}^{(n-2)} \xrightarrow{\partial_{n-2}} \\
&\cdots \xrightarrow{\partial_{2}} \bigoplus_{\substack{J'' \subset \{0, \cdots, n\}, \\ |J''|=2}} \mathfrak{C}_{J''}^{(1)} \xrightarrow{\partial_{1}} \bigoplus_{i \in \{0, \cdots, n\}}C \rightarrow 0
\end{split}
\]
is in fact a chain complex that is exact at the first term.
Here, the differentials $\partial_n$ and $\partial_{n-k}, \ 1 \leq k \leq n-1$ are given by
\end{enumerate}
\end{defn}

Using models of $\Delta^n \times C$, we can define higher homotopies of $L_{\infty}[1]$-morphisms (see Definition \ref{dntchh}):
\begin{defn}[$n$-homotopies]
Let $C$ and $C'$ be $L_{\infty}[1]$-algebras and $f_0, \cdots, f_n : C \rightarrow C',$ $L_{\infty}[1]$-morphisms for $n \geq 1.$ Consider a sequence $\vec{J}$ of subsets
\begin{equation}\nonumber
\vec{J} : J_0 \subsetneq J_1 \subsetneq \cdots \subsetneq J_{n-1} \subsetneq \{0, \cdots, n\}
\end{equation}
with $|J_l| = l+1, \ 0 \leq l \leq n-1.$
We say $f_0, \cdots, f_n : C \rightarrow C'$ are $n$-\textit{homotopic} if there exist a model of $\Delta^n \times C',$ say $\mathfrak{C}'^{(n)},$ and an $L_{\infty}[1]$-morphism $h : C \rightarrow \mathfrak{C}'^{(n)}$ such that
\begin{equation}\nonumber
\Eval^{(1)}_{J_0} \circ \cdots \circ \Eval^{(n)}_{J_{n-1}} \circ h = f_j
\end{equation}
for each sequence $\vec{J}$ with $J_0 = \{j\}.$
We call such a map $h$ an \textit{n-homotopy} ($n \geq 1$) of $f_0, \cdots, f_n.$ 
\end{defn}

The main theorem of this paper, concerning a filling condition for quasi-isomorphisms, is stated as follows (see Proposition \ref{pphhe}):

\begin{thm}[Existence of filling homotopies]
Let $f_0, \cdots, f_{n+1} : C_0 \rightarrow C \ (n \geq 0)$ be quasi-isomorphic $L_{\infty}[1]$-morphisms. Suppose that we are given an $n$-homotopy $h_J : C_0 \rightarrow \mathfrak{C}^{(n)}_J$ of $f_{j_0}, \cdots, f_{j_n}$
 for each given $J = \{ j_0 < \cdots < j_n\} \subset \{0, \cdots, n+1\},$ satisfying $\Eval^{(n)}_{J \cap J'} \circ h_J = \Eval^{(n)}_{J \cap J'} \circ h_{J'}$ for two distinct $J$ and $J'.$ Then there exist  a model $\mathfrak{C}^{(n+1)}$ of $\Delta^{n+1} \times C$ and an $(n+1)$-homotopy $\overline{h} : C_0 \rightarrow \mathfrak{C}^{(n+1)}$ of $f_0, \cdots, f_{n+1}$ such that ${\mathfrak{C}^{(n)}_J}$'s belong to the data for $\mathfrak{C}^{(n+1)},$ satisfying $\mathrm{Eval}^{(n+1)}_J \circ \overline{h} = h_J.$
\end{thm} 

Applying an induction argument on $n$ to the preceding theorem, we obtain (see Corollary \ref{anhp}):

\begin{thm}[Existence of filling homotopies for quasi-isomorphisms]\label{adfasd}
Arbitrarily given quasi-isomorphic $L_{\infty}[1]$-morphisms $f_0, \cdots, f_n : C \rightarrow C' \ (n \geq 1)$ are $n$-homotopic.
\end{thm}

In this paper, we always work over a field, where a strong homotopy version of the Whitehead theorem can be stated and proved following \cite{FOOO1} (see Theorem \ref{wht}):

\begin{thm}[Whitehead theorem]\label{wht1}
Over a field and for strict $L_{\infty}[1]$-algebras, a quasi-isomorphism is a homotopy equivalence.
\end{thm}

The motivation for this work is to make use of the aforementioned structures in developing a new, categorical definition of Kuranishi spaces \cite{Kim1}, \cite{Kim2}. We remark that Theorem \ref{wht1} plays an essential role therein. It is through the notion of higher homotopies together with Theorem \ref{adfasd} that the cocycle-like conditions for Kuranishi spaces are formulated, contributing to a construction of more flexible geometric objects compared to the strict and rigid existing definition in \cite{FOOO2}.

\subsection{Outline of the paper}

We outline the structure of this paper. In Section 2, we briefly recall the definition of $L_{\infty}[1]$-algebras and their morphisms. Section 3 introduces the Fukaya-Oh-Ohta-Ono homotopy in the $L_{\infty}[1]$-setting and proves the Whitehead theorem.In Section 4, we develop a version of $L_{\infty}[1]$-homotopy theory in higher degrees and show that a filling condition holds for simplices whose vertices are assigned quasi-isomorphisms.

\subsection*{Acknowledgment}  This work was supported by the BK21 FOUR program funded by the Ministry of Education (MOE), Korea (No. 4120240414885).

\section{$L_{\infty}[1]$-algebras}

In this section, we briefly introduce $L_{\infty}[1]$-algebras mainly to fix the notation. We first recall the notion of \textit{graded symmetric algebra} $S C$ of a vector space $C$ over a field $\mathbf{k},$
$$
S C := TC/ \langle v \otimes v' - (-1)^{|v|\cdot|v'|} v' \otimes v \rangle,
$$
with its degree $k$ component $S^k C := \{v \in S C \mid v \text{ is homogeneous of degree } k\}.$
We have a decomposition
$$
S C = \bigoplus\limits_{k=0}^{\infty} S^k C
$$
with the induced product $\odot$ on each component.
We denote by $\text{Sh}(i, k-i)$ the set of $(i, k-i)$-unshuffles, and the sign $\sgn(\tau)$ for $\tau \in \text{Sh}(i, k-i)$ is defined for homogeneous elements $a_1, \cdots, a_k \in C$, we write
$$
a_{\tau(1)} \odot \cdots \odot a_{\tau(k)} = \sgn(\tau) a_1 \odot \cdots \odot a_k.
$$

\begin{defn}
An \textit{$L_{\infty}[1]$-algebra} is a pair $\left(C, \{l_k\}\right)$ consisting of a vector space $C$ and a family of degree 1 linear maps
$$
l_k : S^k C \rightarrow C, \ k \geq 0,
$$
satisfying the relations
\begin{equation}\label{quadrel}
\sum\limits_{i = 0}^{k} \sum\limits_{\tau \in {\text{Sh}}(i, k-i)} {\sgn(\tau)} l_{k-i+1}\left(l_i(a_{\tau(1)}, \cdots, a_{\tau(i)}),a_{\tau(i+1)}, \cdots, a_{\tau(k)}\right) = 0.
\end{equation}
\end{defn}

\begin{defn}
Let $(C,\{l_k\})$ and $(C', \{l'_k\})$ be two $L_{\infty}[1]$-algebras. An $L_{\infty}[1]$\textit{-algebra morphism}, or simply $L_{\infty}[1]$\textit{-morphism}
\begin{equation}\label{lrel}
f : C \rightarrow C'
\end{equation}
is a family of degree 0 linear maps
$$
f_k : S^kC \rightarrow C', \ k \geq 0,
$$
satisfying the relations
\begin{equation}\label{frel}
\begin{split}
\sum\limits_{i = 0}^{k}& \sum\limits_{\tau \in {\text{Sh}}(i, k-i)} {\sgn(\tau)} f_{k-i+1}\left(l_i(a_{\tau(1)}, \cdots, a_{\tau(i)}),a_{\tau(i+1)}, \cdots, a_{\tau(k)}\right)\\
&= \sum\limits_{\substack{t, j_1, \cdots, j_t \geq 1,\\ j_1 + \cdots + j_t = k}} \sum\limits_{\tau \in S_k}  \frac{\sgn(\tau)}{t! j_1! \cdots j_t!} \  l'_{t}\bigl(f_{j_1}(a_{\tau(1)}, \cdots, a_{\tau(j_1)}), \cdots, \\
& \quad \quad \quad \quad \quad \quad \quad \quad \quad \quad \quad \quad f_{j_t}(a_{\tau(k -(j_1 + \cdots + j_{t-1}))}, \cdots, a_{\tau(k)})\bigr).
\end{split}
\end{equation}
Here, $S_k$ denotes the symmetric group of permutations of $k$ elements. 
\end{defn}

\begin{defn}
For two $L_{\infty}[1]$-morphisms
$$
f : C \rightarrow C', \ g: C' \rightarrow C'',
$$
we define their \textit{composition}
$$
g \circ f : C \rightarrow C''
$$
by a family of linear maps of degree 0 for $k \geq 0$
\begin{equation}\nonumber
\begin{split}
(g \circ f)_k := &\sum\limits_{i = 0}^{k} \sum\limits_{\tau \in S_k}  \frac{\sgn(\tau)}{t! j_1! \cdots j_t!} \ g_{t}\bigl(f_{j_1}(a_{\tau(1)}, \cdots, a_{\tau(j_1)}), \cdots,\\
&\quad \quad\quad \quad\quad \quad\quad \quad f_{j_t}(a_{\tau(k -(j_1 + \cdots + j_{t-1}))}, \cdots, a_{\tau(k)})\bigr).
\end{split}
\end{equation}
It is straightforward to verify that $\{(g \circ f)_k\}_{k \geq 0}$ satisfies the relation (\ref{frel}).
\end{defn}

\begin{defn}
We say an $L_{\infty}[1]$-algebra $\{l_k\}_{k \geq 0}$ is \textit{strict} if $l_0 = 0.$ Otherwise, we say it is \textit{curved.} We similarly define \textit{strict/curved} $L_{\infty}[1]$-morphisms.
\end{defn}

In the strict case, the relations (\ref{lrel}) and (\ref{frel}) coincide with the differential and the chain map relations, respectively. That is, they satisfy
$$
\begin{cases}
l_1 \left( l_1(a) \right) = 0,\\
l'_1 \left(f_1 (a) \right) = f_1 \left( l_1(a) \right).
\end{cases}
$$
\begin{defn}
We say that a strict $L_{\infty}[1]$-algebra $(C,\{l_k\})$ is \textit{acyclic} if its cohomology for each degree vanishes, that is, if 
\[
H^*(C) = \frac{\ker{l_1}}{\mathrm{Im}l_1} = 0.
\]
We say that a strict $L_{\infty}[1]$-morphism $\{f_k\}_{k \geq 1}$ between strict $L_{\infty}[1]$-algebras is a \textit{quasi-isomorphism} if $f_1$ is a quasi-isomorphic chain map.
\end{defn}

$L_{\infty}[1]$-algebras can be equivalently described within the framework of coalgebras.

\begin{defn}
We say that the vector space $\mathscr{C}$ is a \textit{coalgebra} if it is equipped with the following two linear maps  
\begin{equation}
\begin{cases}
\Delta : \mathscr{C} \rightarrow \mathscr{C} \otimes  \mathscr{C},\\
\varepsilon : \mathscr{C} \rightarrow \mathbf{k},
\end{cases}
\end{equation}
called \textit{comultiplication} and \textit{counit}, respectively. 
We require them to satisfy
\begin{enumerate}[label = (\roman*)]
\item $(\mathrm{id}_{\mathscr{C}} \otimes \Delta) \circ \Delta = (\Delta \otimes \mathrm{id}_{\mathscr{C}}) \circ \Delta,$
\item $(\mathrm{id}_{\mathscr{C}} \otimes \varepsilon) \circ \Delta = (\varepsilon \otimes \mathrm{id}_{\mathscr{C}}) \circ \Delta.$
\end{enumerate}

In our case of graded symmetric algebra $\mathscr{C},$ $\Delta$ is given by 
\[
\Delta : a_1 \odot \cdots \odot a_k \mapsto \sum\limits_{i = 1}^{k-1} \sum\limits_{\tau \in {Sh}(i, k-i)} {\sgn(\tau)} a_{\tau(1)} \odot \cdots \odot a_{\tau(i)} \otimes a_{\tau(i+1)}\odot \cdots \odot a_{\tau(k)} = 0,
\]
while $\varepsilon$ is by the projection to $k = 0$ component.

A coalgebra $\mathscr{C}$ is said to be \textit{coassociative} if
$$
(\mathrm{id}_{\mathscr{C}} \otimes \Delta) \circ \Delta = (\Delta \otimes \mathrm{id}_{\mathscr{C}}) \circ \Delta,
$$
and \textit{cocommutative} if
$$
S \circ \Delta = \Delta,
$$
where the map ${S} : \mathscr{C} \otimes \mathscr{C} \rightarrow \mathscr{C} \otimes \mathscr{C}$ is given by ${S}(a \otimes b) = (-1)^{|a| \cdot |b|} b \otimes a.$
\end{defn}

It is straightforward to verify the following lemma:
\begin{lem}
$(S C ,\Delta, \varepsilon)$ is a cocommutative, coassociative coalgebra.
\end{lem}

To describe $L_{\infty}[1]$-algebras using coalgebras, we introduce coderivations.
\begin{defn}
A coderivation is a degree 1 linear map
\[
d :  \mathscr{C} \rightarrow  \mathscr{C},
\]
satisfying the condition
\[
d \circ \Delta = (d \otimes \mathrm{id}_{\mathcal{C}} + \mathrm{id}_{\mathcal{C}} \otimes d) \circ \Delta.
\]
We say that a coderivation $d$ is a \textit{codifferential} if it further satisfies $d \circ d = 0.$
\end{defn}

\begin{lem}\label{aglem}
An $L_{\infty}[1]$-algebra structure on $C$ uniquely determines a cocommutative, coassociative coalgebra structure on $S C$ equipped with a codifferential.
\end{lem}

\begin{proof}
Each linear map $l_k : S^k C \rightarrow C$ induces a map
$$
\widehat{l}_k : S C \rightarrow S C
$$
given by
$$
\widehat{l}_k(a_1 \odot \cdots \odot a_k) := \sum\limits_{\sigma \in \text{Sh}(i, k -i)} \text{\sgn}(\sigma)l_k(a_{\sigma(1)}, \cdots ,a_{\sigma(i)}) \odot a_{\sigma(i+1)} \odot \cdots \odot a_{\sigma(k)}.
$$

For each component $a_1 \odot \cdots \odot a_k \in S^k C,$ we formally denote
\[
\widehat{l} := \widehat{l}_1 + \widehat{l}_2 + \cdots : SC \rightarrow SC,
\]
which is defined for each $k$ by
$$
\widehat{l} (a_1 \odot \cdots \odot a_k) :=  \sum\limits_{i =1}^{k} \sum\limits_{\sigma \in \text{Sh}(i, k -i)} \text{\sgn}(\sigma)l_k(a_{\sigma(1)}, \cdots ,a_{\sigma(i)}) \odot a_{\sigma(i+1)} \odot \cdots \odot a_{\sigma(k)},
$$
$\widehat{l}$ can be readily verified to be a codifferential on $SC.$

\end{proof}

\section{Homotopies of $L_{\infty}[1]$-morphisms}

In this section, we state an $L_\infty[1]$-algebra version of \cite[Section 4.2]{FOOO1}, and do not claim much originality.

\subsection{Homotopies of $L_\infty[1]$-morphisms}
The material of this subsection is largely a rewriting of \cite[Section 4.2]{FOOO1} in the $L_\infty[1]$-framework, introducing homotopy of $L_{\infty}[1]$-morphisms and related results.

\begin{defn}[Models of $\Delta^1 \times C$]\label{mdl1}
Let $C$ be an $L_{\infty}[1]$-algebra. We say an $L_{\infty}[1]$-algebra $\mathfrak{C}$ is a \textit{model} of $\Delta^1 \times C$ if there exist $L_{\infty}[1]$-morphism
\[
\Eval_j : \mathfrak{C} \rightarrow C, \ \ j = 0,1\\
\]
and
\[
\Incl : {C} \rightarrow \mathfrak{C},
\]
with the following properties:
\begin{enumerate}[label = (\roman*)]
\item $\left(\Eval_j\right)_{k \geq 2} \equiv 0, \ j =0,1, \ \Incl_{k \geq 2} \equiv 0.$
\item $\Eval_j, \ j =0,1$ and  $\Incl$ are quasi-isomorphisms.
\item $(\Eval_j)_1 \circ \Incl = \mathrm{id}_C.$
\item $(\Eval_0)_1 \oplus (\Eval_{1})_1 : \mathfrak{C} \rightarrow {C} \oplus C$ is \textit{surjective.}
\end{enumerate}
\end{defn}

Using the notion of models, we can define homotopies between $L_{\infty}[1]$-morphisms:

\begin{defn}[Homotopy]\label{defn:homotopy}
We say that two $L_{\infty}[1]$-morphisms $f_0, f_1 : (C, \{l_k\}) \rightarrow (C', \{l'_k\})$ are \textit{homotopic} if there exist a \textit{model of} $\Delta^1 \times C'$ denoted by $\mathfrak{C}'$ and an $L_{\infty}[1]$-morphism $h: C \rightarrow \mathfrak{C}'$ such that we have $f_j = \Eval_j \circ h, \ j = 0,1.$
\end{defn}

\begin{lem}
Homotopies define an equivalence relation.
\end{lem}
\begin{proof}
Reflexivity and symmetry trivially hold. For transitivity, let
$h_1 : C \rightarrow \mathfrak{C}'_1$ and $h_2 : C \rightarrow \mathfrak{C}'_2$ be homotopies.
Then we define
\begin{equation}\nonumber
h' : C \rightarrow \overline{\mathfrak{C}}'
\end{equation}
by
\begin{equation}\nonumber
h'(x) := \big( h_1(x), h_2(x) \big),
\end{equation}
where $\overline{\mathfrak{C}}'$ is a model of $\Delta^1 \times C'$
given by
\begin{equation}\nonumber
\overline{\mathfrak{C}}' := \{ (x,y) \in \mathfrak{C}_1 \times \mathfrak{C}_2 \mid \text{deg}x = \text{deg}y, (\Eval_1)_1 x = (\Eval_0)_1y \},
\end{equation}
The $L_{\infty}[1]$-structure $\{\overline{l}_k\}$ on $\overline{\mathfrak{C}}'$ is given by
\begin{equation}\label{pluslinfty}
\overline{l}_k\big((x_1, y_1), \cdots, (x_k, y_k)\big) :=  \big( l_k(x_1, \cdots, x_k), l'_k(y_1, \cdots, y_k) \big),
\end{equation}\nonumber
and the maps Eval$_j, \ j =0,1$ and Incl by
\begin{equation}\label{pluseval}
\begin{split}
&(\Eval_{0})_1(x,y) =  \big( (\Eval_{0})_1(x),(\Eval_{0})_1(y) \big), \\
& (\Eval_{1})_1(x,y) =  \big( (\Eval_{1})_1(x),(\Eval_{1})_1(y) \big),\\
&\Incl_1(x) = \big(\Incl_1(x), \Incl_1(x)\big).
\end{split}
\end{equation}
\end{proof}

\begin{lem}\label{pluslemma}
Let $f, g : C_1 \rightarrow C_2$ and $f', g' : C'_1 \rightarrow C'_2$ be homotopic pairs of $L_{\infty}[1]$-morphisms. Then $f \oplus g$ is homotopic to $f' \oplus g'$ as $L_{\infty}[1]$-morphisms from $C_1 \oplus C_1'$ to $C_2 \oplus C'_2.$
\end{lem}

\begin{proof}
Let $\mathfrak{C}_2$ and $\mathfrak{C}_2'$ be models of $\Delta^1 \times C_2$ and $\Delta^1 \times C_2',$ respectively. Let $h: C_1 \rightarrow \mathfrak{C}_2$ and $h' : C'_1 \rightarrow \mathfrak{C}'_2$ be the homotopies from $f$ to $g$ and from $f'$ to $g',$ respectively. For the desired homotopy, we can take $h \oplus h' : C_1 \oplus C_1' \rightarrow \mathfrak{C}_2 \oplus \mathfrak{C}_2',$ where $L_{\infty}[1]$-structures on both sides are given by (\ref{pluslinfty}).
\end{proof}

We introduce  homotopy equivalence for $L_\infty[1]$-algebras.

\begin{defn}
An $L_{\infty}[1]$-morphism $f : C \rightarrow C'$ is a \textit{homotopy equivalence} if there exists another $L_{\infty}[1]$-morphism $g : C' \rightarrow C$ such that $g \circ f$ and  $f \circ g$ are homotopic to $\mathrm{id}_C$ and $\mathrm{id}_{C'},$ respectively.
\end{defn}

The following is the main ingredient for our purpose.

\begin{thm}[Compare with Theorem 4.2.34 \cite{FOOO1}]\label{th}
Let $C_i, \ i =1,2$ be $L_{\infty}[1]$-algebras and $f : C_1 \rightarrow C_2$ an $L_{\infty}[1]$-morphism. For $\mathfrak{C}_i,$ models of $\Delta^1 \times C_i, \ i =1,2,$ respectively, there exists an $L_{\infty}[1]$-morphism $\mathfrak{F} : \mathfrak{C}_1 \rightarrow \mathfrak{C}_2$ that is \textit{over} $f$ and compatible with $\mathrm{Eval}_j, \ j=0,1$ and $\mathrm{Incl}$ in the following sense:
\begin{enumerate}[label = (\roman*)]
\item $\mathrm{Eval}_{s=j} \circ \mathfrak{F} = f \circ \mathrm{Eval}_{s=j}, \ j =0,1,$\\
\item $\mathrm{Incl} \circ f = \mathfrak{F} \circ \mathrm{Incl.}$
\end{enumerate}
\end{thm}
The statement of the theorem can be visualized into the following
diagram
\[
\begin{tikzcd}
C_1 \arrow[swap]{d}{f} \arrow{r}{\Incl} &  \mathfrak{C}_1 \arrow[swap]{d}{\mathfrak{F}} \arrow{rr}{\Eval_{s=0} \oplus \Eval_{s=1}} & & C_1 \oplus C_1 \arrow{d}{f \oplus f} \\
C_2 \arrow{r}{\Incl} & \mathfrak{C}_2  \arrow{rr}{\Eval_{s=0} \oplus \Eval_{s=1}} & & C_2 \oplus C_2.
\end{tikzcd}
\]

The proof of the theorem will be provided in Appendix B. Meanwhile, we state its immediate consequence.

\begin{prop}\label{prop:heer}
Homotopy equivalences define an equivalence relation.
\end{prop}
\begin{proof}
Only transitivity is non-trivial, and it follows from the following lemma.

\begin{lem}\label{eerel}
Consider a diagram of $L_{\infty}[1]$-algebras and $L_{\infty}[1]$-morphisms.
\begin{equation}\nonumber
C_0 \xrightarrow{f,g} C_1 \xrightarrow{f',g'} C_2.
\end{equation}
If $f \sim g$ and $f' \sim g',$ then we have $f' \circ f \sim g' \circ g.$
\end{lem}
\begin{proof}
By Theorem \ref{th}, we have an $L_{\infty}[1]$-morphism $\mathfrak{F}' : \mathfrak{C}_2 \rightarrow \mathfrak{C}_3$ between a model of $\Delta^1 \times C_2$ and a model of $\Delta^1 \times C_3.$ For a homotopy $\mathfrak{F} : C_1 \rightarrow \mathfrak{C}_2$ between $f$ and $g,$ $\mathfrak{F}' \circ \mathfrak{F}$ is the desired homotopy from $f' \circ f$ to $g' \circ g.$
\end{proof}
The proof of Proposition \ref{prop:heer} follows easily from Lemma \ref{eerel} and Definition \ref{defn:homotopy}.
\end{proof}

\subsection{Whitehead theorem}
In this subsection, we prove an $L_{\infty}[1]$-version of the Whitehead theorem (over a field) that plays a crucial role in our subsequent discussions regarding quasi-isomorphisms. (See \cite[Subsection 4.5]{FOOO1} for its $A_{\infty}$-version.)

We begin with the definitions of $L_{K}[1]$-algebras and $L_{K}[1]$-morphisms.
\begin{defn}
\textit{$L_{K}[1]$-algebras and $L_{K}[1]$-morphisms} are defined by the families $\{l_k\}_{k \leq K}$ and $\{f_k\}_{k \leq K}$ with the same conditions (\ref{lrel}) and (\ref{frel}), respectively.
\end{defn}
Let $C_i, \ i =1,2$ be $L_{K+1}[1]$-algebras and $f : C_1 \rightarrow C_2$ an $L_{K}[1]$-morphism. We consider the space $\mathrm{Hom}\left(S^{K+1}C_1, C_2\right)$ with the Hochschild differential $\delta_1$ given by $\delta_1(\cdot) := l_1 \circ (\cdot) + (-1)^{\text{deg}(\cdot) +1}(\cdot) \circ \widehat{l}_1.$ Here $\widehat{l}_1 : SC_1 \rightarrow SC_1$ is the coderivation induced by $l_1$ on $SC_1.$

We denote
\begin{equation}\label{skcn}
S^{\leq K+1} C := \bigoplus\limits_{i=1}^{K+1} S^i C,
\end{equation}
and note that $f$ induces
\begin{equation}\nonumber
\widehat{f}_{\leq K}:= \sum\limits_{i = 1}^{K}\widehat{f}_i \in \mathrm{Hom}\left(S^{\leq K+1}C_1, S^{\leq K+1}C_2\right).
\end{equation}
Also we denote
\[
\widehat{l}_{\leq K+1} := \sum\limits_{i = 1}^{K+1}\widehat{l}_i.
\]

We define the $(K+1)$\textit{-th obstruction class of} $f$ to the following \textit{degree 1} element:
\begin{equation}\nonumber
O_{K+1}(f) := \widehat{l}_{\leq K+1} \circ \widehat{f}_{\leq K} - \widehat{f}_{\leq K} \circ \widehat{l}_{\leq K+1} \in \mathrm{Hom}\left(S^{\leq K+1}C_1,  S^{\leq K+1}C_2\right).
\end{equation}

\begin{lem}\label{fasgg}
In the above situation, $O_{K+1}(f)$ satisfies the following properties:
\begin{enumerate}[label = (\roman*)]
\item $O_{K+1}(f)|_{S^{\leq K}C_1} = 0.$
\item $\mathrm{Im} \left(O_{K+1}(f)\right) \subset C_2.$
\item $\delta_1 \big( O_{K+1}(f) \big) = 0.$
\item $[O_{K+1}(f)] = 0$ if and only if there exists an $L_{K+1}[1]$-morphism that extends $f.$
\item For $L_{K+1}[1]$-morphisms $g: C_1' \rightarrow C_1$ and $g' : C_2 \rightarrow C_2',$ we have 
\[
[O_{K+1}(g' \circ f \circ g)] = (g_1')_* \circ [O_{K+1}(f)] \circ (S^{\leq K+1}g_1)_*,
\]
 where $S^{\leq K+1}g_1 : S^{\leq K+1}C_1' \rightarrow S^{\leq K+1}C_1$ is induced from $g_1$ and $(S^{\leq K+1}g_1)_*$ is the map induced on cohomology.
\item If $f$ is homotopic to $f',$ then we have $[O_{K+1}(f)] = [O_{K+1}(f')].$
\end{enumerate}
\end{lem}

\begin{proof}
(i) amounts to saying that $f$ is an $L_{K}[1]$-morphism. (ii) follows immediately from the assumption that $f$ is an $L_{K}[1]$-morphism and the definition of $O_{K+1}(f).$ For (iii), we note that we have
\[
\delta_1 \big( O_{K+1}(f) \big)  = l_1 \circ O_{K+1}(f) - O_{K+1}(f) \circ \widehat{l}_1 = 0,
\]
which essentially follows from the fact $f$ is an $L_{K}[1]$-morphism and a straightforward computation. For (iv), observe that $[O_{K+1}(f)]$ vanishes if and only if there exists $f_{K+1}$ such that  $-\delta_1(f_{K+1}) = O_{K+1}(f),$ which is precisely the relation that $f_{K+1}$ together with $f$ must satisfy to be an $L_{K+1}[1]$-morphism. (v) can be verified straightforwardly. For (vi), let $h$ be an $L_{K}[1]$-homotopy (arising from a model of $\Delta^1 \times C_2$) between $f$ and $f'.$ Then we have
\begin{equation}\nonumber
\begin{split}
[O_{K+1}(f)] &= [O_{K+1}(\Eval|_{s=0} \circ h)] \overset{(1)}{=} \left((\Eval|_{s=0})_1\right)_*[O_{K+1}(h)]\\ 
& \overset{(2)}{=} \left((\Eval|_{s=1})_1\right)_*[O_{K+1}(h)] \overset{(3)}{=} [O_{K+1}(\Eval|_{s=1} \circ h)] = [O_{K+1}(f')],
\end{split}
\end{equation}
where the equalities $(1)$ and $(3)$ follow from (v). The equality $(2)$ follows from the axiom (iii) of Definition \ref{mdl1}, stating that $\mathrm{Eval}_{s=j} \circ \mathrm{Incl}=\mathrm{id}_{C_2}, \ j=0,1$ and that they are quasi-isomorphisms.
\end{proof}

\begin{cor}\cite[Corollary 4.5.5]{FOOO1}
Let $f : C_1 \rightarrow C_2$ be an $L_{K+1}[1]$-morphism, $g: C_1 \rightarrow C_2$ an $L_{K}[1]$-morphism and $h : C_1 \rightarrow \mathfrak{C}_2$ an $L_{K}[1]$-homotopy from $f$ to $g.$ Then $g$ extends to an $L_{K+1}[1]$-morphism $g',$ and $h$ extends to an $L_{K+1}[1]$-homotopy from $f$ to $g.$
\end{cor}

\begin{proof}
We have 
\[
(\mathrm{Eval}_{s=0})_{1*}[O_{K+1}(h)] = [O_{K+1}(\Eval_{s=0} \circ h)] = [O_{K+1}(f)], \ j = 0,1
\]
by (v) of the previous lemma. Then from the fact that $f$ is an $L_{K+1}[1]$-morphism and that $\mathrm{Eval}_{s=j}$ is a quasi-isomorphism, we obtain $[O_{K+1}(h)] = 0,$ that is, $h$ extends to an $L_{K+1}[1]$-morphism.

Now we denote $h'_{K+1} := \Incl \circ f_{K+1}$ and observe that
\begin{equation}\nonumber
\begin{split}
(\Eval_{s=0})_1 \big( O_{K+1}(h) &+ \delta_1 (h'_{K+1}) \big) = (\Eval_{s=0})_1 \left( O_{K+1}(h)\right) +  (\Eval_{s=0})_1 \left( \delta_1 (h'_{K+1})\right)\\ 
& \overset{*}{=} \left( O_{K+1}((\Eval_{s=0}) \circ h)\right) + \left((\Eval_{s=0})_1 \circ \Incl_1\right) \delta_1 (f_{K+1}) \\ 
& = O_{K+1}(f) +  \delta_1 (f_{K+1}) = 0,
\end{split}
\end{equation}
as $f$ is an $L_{K+1}[1]$-morphism. For $*,$ we use (i) and (iii) of Definition \ref{mdl1}. Here we regard $(\Eval_{s=0})_j$'s as maps from $\mathrm{Hom}\left(S^{\leq K+1}C_1, \mathfrak{C}_2\right)$ to $\mathrm{Hom}\left(S^{\leq K+1}C_1, C_2\right),$ and similarly for $\Incl.$

Since $(\mathrm{Eval}_{s=0})_1$ is a quasi-isomorphism, we then have
\begin{equation}\nonumber
O_{K+1}(h) + \delta_1 (h'_{K+1}) = \delta_1(\Delta h_{K+1})
\end{equation}
for some $\Delta h_{K+1} \in \mathrm{ker}(\Eval_{s=0})_1.$ Further denoting $h_{K+1}:=h'_{K+1} - \Delta h_{K+1},$ we verify that $h_1, \cdots, h_{K+1}$ define an $L_{K+1}[1]$-morphism, $\overline{h}.$ Moreover, $g':= \Eval_{s=1} \circ \overline{h}$ is an $L_{K+1}[1]$-morphism that extends $g$ by the induction hypothesis. We conclude that $\overline{h}$ is the desired $L_{K+1}[1]$-homotopy from $f$ to $g'.$
\end{proof}

\begin{prop}\label{prevprop}\cite[Proposition 4.5.6]{FOOO1}
Let $f : C_1 \rightarrow C_2$ be an $L_{\infty}[1]$-quasi-isomorphism, $g^{(K)}: C_2 \rightarrow C_1$ an $L_K[1]$-morphism, and $h^{(K)} : C_1 \rightarrow \mathfrak{C}_1$ an $L_K[1]$-homotopy from identity $\mathrm{id}_{C_1}$ to $g^{(K)} \circ f.$ Then $g^{(K)}$ extends to an $L_{K+1}[1]$-morphism $g^{(K+1)},$ and $h^{(K)}$ extends to an $L_{K+1}[1]$-homotopy $h^{(K+1)}$ from $\mathrm{id}_{C_1}$ to $g^{(K+1)} \circ f.$
\end{prop}

\begin{proof}
Since $\mathrm{id}_{C_1, k \geq 2} \equiv 0,$ by the previous corollary, we have $h_{K+1}'$ such that $h_1, \cdots, h_K,$ and $h_{K+1}'$ form an $L_{K+1}[1]$-morphism ${h}' : C_1 \rightarrow \mathfrak{C}_1$ and that $\Eval_{s=0} \circ h' = \mathrm{id}_{C_1}.$

From the definition of the obstruction class $O_{K+1}(\cdot),$ it follows that
\begin{equation}\label{ost}
O_{K+1}\left(g^{(K)} \circ f\right) = - \delta_1\left((\Eval_{s=1})_1 \circ h_{K+1}'\right).
\end{equation}
In other words, $(\Eval_{s=1})_1 \circ h_{K+1}'$ extends $g^{(K)}\circ f$ to an $L_{K+1}[1]$-morphism. Since $f$ is a quasi-isomorphism, we then have $[O_{K+1}\left(g^{(K)}\right)] = 0$ from (v) of Lemma \ref{fasgg}. Hence there exists $g'_{K+1} \in \mathrm{Hom}\left(S^{K+1}C_2, C_1\right)$ such that
\begin{equation}\label{gagaga}
O_{K+1}\left(g^{(K)}\right) = -\delta_1\left(g'_{K+1}\right).
\end{equation}

We denote
\[
\Xi := g'_{K+1} \circ f_1^{\otimes K+1} - (\Eval_{s=1})_1 \circ h'_{K+1} \in \mathrm{Hom}\left(S^{K+1}C_1, C_1\right).
\]
Observe that (\ref{ost}) and (\ref{gagaga}) imply that $\delta_1(\Xi)=0.$

Then by the quasi-isomorphism of $f,$ there exists a $\delta_1$-cocycle 
\[
\Delta g'_{K+1} \in \mathrm{Hom}\left(S^{K+1}C_2, C_1\right)
\]
so that $g'_{K+1}$ replaced by $g'_{K+1} + \Delta g'_{K+1}$ still satisfies (\ref{gagaga}). For this choice, there exist $\delta_1\left(g''_{K+1}\right)$-worth choices involved for some $g''_{K+1} \in \mathrm{Hom}\left(S^{K+1}C_2, C_1\right)$. Then, again by the quasi-isomorphism of $f$, it follows that $\left[\Xi + \Delta g'_{K+1} \circ f_1^{\otimes K+1}\right] = 0$. In other words, there exists
\[
\Delta_1 h_{K+1} \in \mathrm{Hom}\left(S^{K+1}C_1, C_1\right),
\]
satisfying
\[
\delta_1\left(\Delta_1 h_{K+1}\right) = \left(g_{K+1}' + \Delta g_{K+1}'\right) \circ f_1^{\otimes K+1} - \left(\mathrm{Eval}_{s=1}\right)_1 \circ h'_{K+1}.
\]
Since $\bigoplus\limits_{j} (\Eval_{s=j})_1$ is surjective by (iv) of Definition \ref{mdl1}, we obtain 
\[
\Delta h_{K+1} \in \mathrm{Hom}\left(S^{K+1}, \mathfrak{C}_1\right)
\]
such that
\[
\left(\Eval_{s=0}\right)_1 \circ \Delta h_{K+1} = 0, \ \left(\Eval_{s=1}\right)_{1} \circ \Delta h_{K+1} = \Delta_1 h_{K+1}.
\]
Now denoting
\[
g_{K+1} := g'_{K+1} + \Delta g_{K+1}' \text{ and }\ h_{K+1}:= h'_{K+1} + \delta_1 \left(\Delta h_{K+1}\right),
\]
we can easily show that $g_1, \cdots, g_{K+1}$ and $h_1, \cdots, h_{K+1}$ define $L_{K+1}[1]$-morphisms (denoted by $g^{(K+1)}$ and $h^{(K+1)}$) that extend $g^{(K)}$ and $h^{(K)},$ respectively.
Moreover, it immediately follows that $h^{(K+1)}$ is an $L_{K+1}[1]$-homotopy from identity to $g^{(K+1)} \circ f.$
\end{proof}

The following theorem is an $L_{\infty}[1]$-algebra version of \cite[Theorem 4.2.45 (1)]{FOOO1}.
\begin{thm}[Whitehead theorem]\label{wht}
Over a field, a quasi-isomorphic $L_{\infty}[1]$-morphism is a homotopy equivalence.
\end{thm}

\begin{proof}
Let $f : C_1 \rightarrow C_2$ be a quasi-isomorphic $L_{\infty}[1]$-morphism. Recall that for chain complexes over a field, quasi-isomorphism is equivalent to chain homotopy equivalence (cf. \cite[Remark 2.9]{AT}). Moreover, chain homotopy equivalence coincides with $L_{1}[1]$-homotopy equivalence (cf. \cite[Lemma 2.4]{AT}). Thus, there exists a chain map $g_1 : C_2 \rightarrow C_1$ such that $g_1 \circ f_1$ is chain homotopic to identity. Denote by $g^{(1)}$ the $L_1[1]$-morphism interpretation of $g_1$ (with the trivial higher-order operations) and by $h'_1$ the corresponding chain homotopy.

Since $\bigoplus\limits_{i} \left(\Eval_{s=i}\right)_1$ is surjective, we have $h_1'' : C_1 \rightarrow \mathfrak{C}_1$ such that
\begin{equation}\nonumber
\left(\Eval_{s=0}\right)_1 \circ h_1'' = 0, \ \left(\Eval_{s=1}\right)_1 \circ h_1'' = h_1',
\end{equation}
where $\mathfrak{C}_1$ is a model of $\Delta^1 \times C_1.$ We then denote 
\[
h_1 := \Incl_1 + l_1 \circ h_1'' + h_1'' \circ l_1 : C_1 \rightarrow \mathfrak{C}_1 
\]
and verify that it is an $L_1[1]$-homotopy from the identity morphism $\mathrm{id}_{C_1}$ to $g^{(1)} \circ f$:
\begin{equation}
\begin{split}
(\Eval_{s=0})_1 \circ h_1 &= (\Eval_{s=0})_1 \circ (\Incl_1 + l_1 \circ h_1'' + h_1'' \circ l_1) = \mathrm{id}_{C_1},\\ 
(\Eval_{s=1})_1 \circ h_1 &= (\Eval_{s=1})_1 \circ (\Incl_1 + l_1 \circ h_1'' + h_1'' \circ l_1) \\
&= \mathrm{id}_{C_1} + g^{(1)} \circ f - id_{C_1} = g_1 \circ f_1.
\end{split}
\end{equation}

Applying Proposition \ref{prevprop} inductively, we then obtain an $L_{\infty}[1]$-morphism $g : C_2 \rightarrow C_1$ and an $L_{\infty}[1]$-homotopy $h$ from the identity morphism to $g \circ f.$

Similarly, there exists $f'$ such that $f' \circ g$ is homotopic to identity. Observe that $f \sim f' \circ g \circ f \sim f',$ so that $\mathrm{id}_{C_2} \sim f' \circ g \sim f \circ g.$ Thus $g$ is the desired homotopy inverse of $f.$
\end{proof}

\section{Higher homotopies of $L_{\infty}[1]$-morphisms}

In this section, we develop a system of models which enables us to uniformly handle the homotopy of homotopies and general higher homotopies (cf. \cite[Remark 7.2.262]{FOOO1}). The lower-degree analogues of $A_{\infty}$-structures are explained in detail in \cite[Section 4.2]{FOOO1} ($n=1$) and \cite[Section 7.2.12]{FOOO1} ($n=2$).

\subsection{Models of $\Delta^n \times C$ with $n \geq 2$}

We claim that the following definition provides a systematic uniform higher degree simplicial extension of the models of $\Delta^k \times C$ used for $k=1, \, 2$ in \cite[Section 4.2 \& Subsection 7.2.12]{FOOO1}. In particular, the $L_{\infty}[1]$-morphism $\mathrm{Eval}_j$'s ($j = 0,1$) and the chain map $\mathrm{Incl}$ in Definition \ref{mdl1} coincide with $\mathrm{Eval}^{(1)}_j$'s ($j = 0,1$) and $\mathrm{Incl}^{(1)}$ in the following definition, respectively.

\begin{defn}[Models of $\Delta^n \times C$]\label{hhtp}
Let $C$ be an $L_{\infty}[1]$-algebra. Suppose that we have defined models of $\Delta^k \times C$ with $k \leq n-1.$ We recursively define \textit{models of} $\Delta^n \times C$ with $n \geq 2$ to be a collection of $L_{\infty}[1]$-algebras
\begin{equation}\nonumber
\mathfrak{C}^{(n)}, \  \left(\mathfrak{C}^{(n)}\right)_{J} \equiv \mathfrak{C}_{J}^{(n-1)}, \text{ where } J \text{ is a subset of } \{0, \cdots, n\} \text{ such that } |J| = n,
\end{equation}
together with an $L_{\infty}[1]$-morphisms
\[
\Eval_J^{(n)} :  \mathfrak{C}^{(n)} \rightarrow \mathfrak{C}_{J}^{(n-1)},
\]
and
\[
\Incl^{(n)} : {C} \rightarrow \mathfrak{C}^{(n)}\\
\]
with the following properties:
\begin{enumerate}[label = (\roman*)]
\item $\mathfrak{C}_{J}^{(n-1)}$ is a model of $\Delta^{n-1} \times C$ with $\mathfrak{C}_{\{i\}}^{(0)} = C$ for each $i.$
\item $\left(\mathfrak{C}_{J}^{(n-1)} \right)_{J'} = \left(\mathfrak{C}_{J'}^{(n-1)} \right)_{J} = \mathfrak{C}_{J \cap J'}^{(n-2)}$ for all $J, J' \subset \{0, \cdots, n\}$ with $|J| = |J'| = n$ and $|J \cap J'| = n-1.$ 
\item $\left(\Eval_j\right)_{k \geq 2} \equiv 0, \ j =0,1, \ \Incl_{k \geq 2} \equiv 0.$
\item $\Eval_J^{(n)}$ and $\Incl^{(n)}$ are quasi-isomorphisms.
\item $\Eval_J^{(n)} \circ \Incl^{(n)} = \Incl_J^{(n-1)},$
where $ \Incl_J^{(n-1)}$ is the Incl map for $\mathfrak{C}_{J}^{(n-1)},$ the model of $\Delta^{n-1} \times C$ for the index $J.$
\item The following sequence of chain complexes
\[
\begin{split}
\mathfrak{C}^{(n)} \xrightarrow{\partial_n}  \bigoplus_{\substack{J \subset \{0, \cdots, n\}, \\ |J| = n}} &\mathfrak{C}_J^{(n-1)} \xrightarrow{\partial_{n-1}} \bigoplus_{\substack{J' \subset \{0, \cdots, n\}, \\ |J'|=n-1}} \mathfrak{C}_{J'}^{(n-2)} \xrightarrow{\partial_{n-2}} \\
&\cdots \xrightarrow{\partial_{2}} \bigoplus_{\substack{J'' \subset \{0, \cdots, n\}, \\ |J''|=2}} \mathfrak{C}_{J''}^{(1)} \xrightarrow{\partial_{1}} \bigoplus_{i \in \{0, \cdots, n\}}C \rightarrow 0
\end{split}
\]
is in fact a chain complex that is exact at the first term. In other words, we require
\[
\ker \partial_{n-1} = \mathrm{Im} \partial_n.
\]
Here, the differentials $\partial_n$ and $\partial_{n-k}, \ 1 \leq k \leq n-1$ are given by
\begin{equation}\label{ppnnk}
\begin{cases}
\partial_n &:= \bigoplus\limits_{\substack{J \subset \{0, \cdots, n\}, \\ |J| = n}} \left(\Eval_J^{(n)}\right)_1, \\
\partial_{n-k} &:= \sum\limits_{\substack{J' \subsetneq J \subset \{0, \cdots, n\}, \\ |J| = |J'| + 1 = n-k}} \partial_{n-k, J, J'},
\end{cases}
\end{equation}
and each $\partial_{n-k, J, J'} : \mathfrak{C}_J^{(n-k)} \rightarrow \mathfrak{C}_{J'}^{(n-k-1)}$ by
\begin{equation}\nonumber
\partial_{n-1, J, J'} := \mathrm{sgn}\big(\sigma(J',J \setminus J')\big) \left(\Eval_{J, J'}^{(n-k)}\right)_1,
\end{equation}
where the map
\begin{equation}\nonumber
\Eval^{(n-k)}_{J, J'} : \mathfrak{C}_{J}^{(n-k)} \rightarrow \mathfrak{C}_{J'}^{(n-k-1)}
\end{equation}
is from the model of $\Delta^{n-k} \times C,$ that is, from $\mathfrak{C}_{J}^{(n-k)},$ while $\sigma(J',J \setminus J')$ denotes the $(J',J \setminus J')$-unshuffle.

In particular, we have 
\[
\left(\Eval^{(n-1)}_{J_1 \cap J_2}\right)_1 \circ \left(\Eval^{(n)}_{J_1}\right)_1 = \left(\Eval_{J_1 \cap J_2}^{(n-1)}\right)_1 \circ \left(\Eval_{J_2}^{(n)}\right)_1
\] 
for all $J_1, J_2 \subset \{0, \cdots, n\}$ with $|J_1| = |J_2| = n$ and $|J_1 \cap J_2| = n-1.$
\end{enumerate}
\end{defn}

Using models of $\Delta^n \times C$, we can define higher homotopies:
\begin{defn}\label{dntchh}
Let $C$ and $C'$ be $L_{\infty}[1]$-algebras and $f_0, \cdots, f_n : C \rightarrow C',$ $L_{\infty}[1]$-morphisms for $n \geq 1.$ Consider a sequence $\vec{J}$ of subsets
\begin{equation}\nonumber
\vec{J} : J_0 \subsetneq J_1 \subsetneq \cdots \subsetneq J_{n-1} \subsetneq \{0, \cdots, n\}
\end{equation}
with $|J_l| = l+1, \ 0 \leq l \leq n-1.$
We say $f_0, \cdots, f_n : C \rightarrow C'$ are $n$-\textit{homotopic} if there exist a model of $\Delta^n \times C',$ say $\mathfrak{C}'^{(n)},$ and an $L_{\infty}[1]$-morphism $h : C \rightarrow \mathfrak{C}'^{(n)}$ such that
\begin{equation}\nonumber
\Eval^{(1)}_{J_0} \circ \cdots \circ \Eval^{(n)}_{J_{n-1}} \circ h = f_j
\end{equation}
for each sequence $\vec{J}$ with $J_0 = \{j\}.$
We call such a map $h$ an $L_{\infty}[1]$-$n$\textit{-homotopy}, or simply \textit{n-homotopy} ($n \geq 1$) of $f_0, \cdots, f_n.$ 
\end{defn}

\begin{rem}
\begin{enumerate}[label = (\roman*)]
\item The $n$-homotopy $h$ in the previous definition is well-defined by the axiom (vi) of Definition \ref{hhtp}, that is, it is independent of the choice of $\vec{J}.$
\item It follows from the definition that if $L_{\infty}[1]$-morphisms $f_0, \cdots f_n$ ($n \geq 2$) are $n$-homotopic, then $f_{j_0}, \cdots, f_{j_m}$ are $m$-homotopic for each tuple $j_0 < \cdots < j_m$ with $\{j_0, \cdots, j_m\} \subset \{0, \cdots, n\}, \ m \leq n.$
\item The previous definition naturally generalizes Definition \ref{mdl1} for the lower degree notion.
\end{enumerate}
\end{rem}

\begin{exam}\label{hhex}
Let $\Delta^n$ be the standard $n$-simplex and $\Omega^*(\Delta^n)$ its de Rham complex over a field. We denote
\begin{equation}\nonumber
\begin{split}
\mathfrak{C}^{(n)} &:= \Omega^*(\Delta^n) \otimes C,\\
\mathfrak{C}^{(n-1)}_{J_i} &:= \Omega^*(\partial_i \Delta^n) \otimes C, \text{ where } \ J_i = \{0, \cdots, n\} \setminus \{i\} \text{ for } 0 \leq i \leq n.\
\end{split}
\end{equation}
On $\mathfrak{C}^{(n)},$ there exists an $L_{\infty}[1]$-algebra structure
\[
l_k : (\mathfrak{C}^{(n)})^{\otimes k} \rightarrow \mathfrak{C}^{(n)}, \ k \geq 1,
\]
which is given by
\begin{equation}\nonumber
l_k(\alpha_1 \otimes x_1, \cdots, \alpha_k \otimes x_k) := 
\begin{cases}
d \alpha_1 \otimes x_1 + (-1)^{|\alpha_1|} \alpha_1 \otimes l_1(x_1) &\text{ if } k =1,\\
 (-1)^{|\vec{\alpha}|} \alpha_1 \wedge \cdots \wedge \alpha_k \otimes l_k(x_1, \cdots, x_k) &\text{ if } k \geq 2,
\end{cases}
\end{equation}
for each $\alpha_i \in \Omega^*(\Delta^n), \ x_i \in C, \ i = 1, \cdots, k,$ and $k \geq 1.$ Here, we denote
\begin{equation}\nonumber
|\vec{\alpha}| := \sum\limits_{i = 1}^{k-1}|x_i| \cdot (|\alpha_{i+1}|+ \cdots + |\alpha_k|) + \sum\limits_{i = 1}^k |\alpha_i|.
\end{equation}

The $L_{\infty}[1]$-morphisms $\Eval^{(n)}_{J_i} = \left\{\left(\Eval^{(n)}_{J_i}\right)_k\right\}_{k \geq 1}$ and $\Incl^{(n)}$ are given by
\begin{equation}\nonumber
\left(\Eval_{J_i}^{(n-1)}\right)_k : \left(\Omega^*(\Delta^n) \otimes C\right)^{\otimes k} \rightarrow \Omega^*(\partial_i \Delta^n) \otimes C,
\end{equation}
\begin{equation}\nonumber
\left(\Eval_{J_i}^{(n-1)}\right)_k :=  
\begin{cases}
\text{(the restriction to $i$-th face)} \otimes \textrm{id}_C & \text{ if } k = 1,\\
 0 &\text{ if } k \geq 2,
\end{cases}
\end{equation}
and
\begin{equation}\nonumber
\begin{split}
\Incl_k^{(n)} &: C^{\otimes k} \rightarrow \Omega^*(\Delta^n) \otimes C,\\
\Incl_k^{(n)} &:= 
\begin{cases}
1 \otimes \textrm{id}_C & \text{ if } k = 1,\\
 0 &\text{ if } k \geq 2,
\end{cases}
\end{split}
\end{equation}
respectively. It immediately follows that all the conditions in Definition \ref{hhtp} are satisfied. In particular, one can easily see that $\Eval^{(n)}_{J_i}$ and $\Incl^{(n)}$ are $L_{\infty}[1]$-algebra morphisms. Moreover, they are quasi-isomorphisms, whose proof can be sketched as follows. Since $\Incl_1^{(n)}$ is injective, the quasi-isomorphicity is equivalent to the acyclicity of the quotient complex 
\[
\frac{\Omega^*(\Delta^n) \otimes C}{\Incl_1^{(n)}(C)} \simeq \frac{\Omega^*(\Delta^n) \otimes C}{\{1\} \otimes C} \simeq \frac{\Omega^*(\Delta^n)}{\{\text{const. ftns.}\}} \otimes C,
\]
which follows from the acyclicity of $\frac{\Omega^*(\Delta^n)}{\{\text{const. ftns.}\}}$ and the K\"{u}nneth formula. Finally, the axiom (v) of Definition \ref{hhtp} with an inductive argument implies that $\Eval^{(n)}_{J_i}$ is also a quasi-isomorphism.
\end{exam}

We now state a key proposition in this section:
\begin{prop}[Existence of filling homotopies]\label{pphhe}
Let $f_0, \cdots, f_{n+1} : C_0 \rightarrow C \ (n \geq 0)$ be quasi-isomorphic $L_{\infty}[1]$-morphisms. Suppose that we are given an $n$-homotopy $h_J : C_0 \rightarrow \mathfrak{C}^{(n)}_J$ of $f_{j_0}, \cdots, f_{j_n}$
 for each given $J = \{ j_0 < \cdots < j_n\} \subset \{0, \cdots, n+1\},$ satisfying 
\begin{equation}\label{adsfdsfaag}
\Eval^{(n)}_{J \cap J'} \circ h_J = \Eval^{(n)}_{J \cap J'} \circ h_{J'} 
\end{equation}
for two distinct $J$ and $J'.$ Then there exist  a model $\mathfrak{C}^{(n+1)}$ of $\Delta^{n+1} \times C$ and an $(n+1)$-homotopy $\overline{h} : C_0 \rightarrow \mathfrak{C}^{(n+1)}$ of $f_0, \cdots, f_{n+1}$ such that ${\mathfrak{C}^{(n)}_J}$'s belong to the data for $\mathfrak{C}^{(n+1)},$ satisfying $\mathrm{Eval}^{(n+1)}_J \circ \overline{h} = h_J.$
\end{prop} 
Before providing its proof in Subsection \ref{spppp}, which is lengthy, we state an immediate consequence.

\begin{cor}\label{anhp}
Arbitrarily given quasi-isomorphic $L_{\infty}[1]$-morphisms $f_0, \cdots, f_n : C \rightarrow C' \ (n \geq 1)$ are $n$-homotopic.
\end{cor}

\begin{proof}
We can proceed with an induction on $n$ with Proposition \ref{pphhe}.
\end{proof}

\subsection{Proof of Proposition \ref{pphhe}}\label{spppp}

In this subsection, we give the proof of Proposition \ref{pphhe}. 

\noindent
\begin{proof}[Proof of Proposition \ref{pphhe}]

Construct the following chain complex from the family $\left\{\mathfrak{C}_{{J}}^{(n)}\right\}_{{J}}$ in (vi) of Definition \ref{hhtp}:
\[
\begin{split}
\bigoplus\limits_{\substack{J \subset \{0, \cdots, n+1\}, \\ |J| = n+1}} \mathfrak{C}_J^{(n)} \xrightarrow{\partial_n}  \bigoplus_{\substack{J' \subset \{0, \cdots, n+1\}, \\ |J'| = n}} &\mathfrak{C}_{J'}^{(n-1)} \xrightarrow{\partial_{n-1}}\\
&\cdots \xrightarrow{\partial_{2}} \bigoplus_{\substack{J'' \subset \{0, \cdots, n+1\}, \\ |J''|=2}} \mathfrak{C}_{J''}^{(1)} \xrightarrow{\partial_{1}} \bigoplus_{i \in \{0, \cdots, n+1\}}C \rightarrow 0,
\end{split}
\]
where the differentials are given in the same way as (\ref{ppnnk}).

Then we add $\ker \partial_n$ to obtain:
\[
\begin{split}
\ker \partial_n \overset{\iota}{\hookrightarrow} \bigoplus\limits_{\substack{J \subset \{0, \cdots, n+1\}, \\ |J| = n+1}} &\mathfrak{C}_J^{(n)} \xrightarrow{\partial_n}  \bigoplus_{\substack{J' \subset \{0, \cdots, n+1\}, \\ |J'| = n}} \mathfrak{C}_{J'}^{(n-1)} \xrightarrow{\partial_{n-1}} \\
& \quad \quad \quad \quad \quad \cdots \xrightarrow{\partial_{2}} \bigoplus_{\substack{J'' \subset \{0, \cdots, n+1\}, \\ |J''|=2}} \mathfrak{C}_{J''}^{(1)} \xrightarrow{\partial_{1}} \bigoplus_{i \in \{0, \cdots, n+1\}}C \rightarrow 0,
\end{split}
\]
where $\iota$ denotes the inclusion map.

Since $\partial_n \circ \bigoplus\limits_{{J}} h_{{J},1} = 0$ follows from (\ref{adsfdsfaag}), the map $\bigoplus\limits_{{J}} h_{{J},1}$ factors through $C_0 \to \ker \partial_n \overset{\iota}{\hookrightarrow} \bigoplus\limits_{{J}} \mathfrak{C}_{{J}}^{(n)}.$ In other words, there exists a chain map
\[
\overset{\circ}{h}_1 : C_0 \rightarrow \ker \partial_n
\]
such that 
\begin{equation}\label{kkklagg}
\iota \circ \overset{\circ}{h}_1 = \bigoplus\limits_{{J}} h_{{J},1}.
\end{equation}

Define the chain complex $\mathfrak{C}^{(n+1)}$ by the mapping cylinder:
\begin{equation}\label{adsgaahahhh}
\mathfrak{C}^{(n+1)} := {\mathrm{Cyl}} \simeq \bigoplus_{m} \big((\ker \partial_n)_m \oplus (\ker \partial_n)_{m+1} \oplus (C_0)_m\big)
\end{equation}
with the differential
\[
(x,y,z) \mapsto \big(dx + y, dy, dz\big).
\]
\[
\]
Recall that the inclusion
\begin{equation}\label{kercy}
i_k : \ker \partial_{n} \hookrightarrow \mathrm{Cyl}
\end{equation}
to the first component is a chain map.

We then define the chain map
\[
\overline{h}_1: C_0 \rightarrow \mathfrak{C}^{(n+1)} := \mathrm{Cyl}
\]
by
\[
\overline{h}_1(w) := \big(\overset{\circ}{h}_1(w), 0, w\big)
\]
for each $w \in C_0$. It follows that $\overline{h}$ is an injective chain map, hence we obtain a short exact sequence of chain complexes
\begin{equation}\label{sescckk}
0 \rightarrow C_0 \xrightarrow{\overline{h}_1} \mathfrak{C}^{(n+1)} \xrightarrow{g_{12}} \bigoplus\limits_m \big(\left(\ker \partial_n\right)_m \oplus \left(\ker \partial_n\right)_{m+1}\big) \rightarrow 0.
\end{equation}
Here, $g_{12}$ denotes the surjective chain map defined by
\[
\begin{split}
g_{12} : \mathfrak{C}^{(n+1)} &\rightarrow \bigoplus\limits_m \big(\left(\ker \partial_n\right)_m \oplus \left(\ker \partial_n \right)_{m+1}\big),\\
(u,v,w) &\mapsto \big(u - \overset{\circ}{h}_1(w), v\big),
\end{split}
\]
and it immediately follows that $\mathrm{Im}\overline{h}_1 = \ker g_{12}.$
Notice that $\big(\left(\ker \partial_n\right)_m \oplus \left(\ker \partial_n\right)_{m+1}\big) $ in (\ref{sescckk}) is nothing but the mapping cone of the identity map $\text{id}_{\ker \partial_n}$ on $\ker \partial_n.$ Since $H^*\big(\text{Cone}(\text{id}_{\ker \partial_n})\big) = 0$, we conclude that $\overline{h}_1$ is a quasi-isomorphism.

Since (the first component) $\ker \partial_n$ in $\bigoplus\limits_m\big(\left(\ker \partial_n\right)_m \oplus \left(\ker \partial_n\right)_{m+1}\big)$ is a direct summand (as a $\mathbf{k}$-module), the inclusion $\iota : \ker \partial_n \rightarrow \bigoplus\limits_J \mathfrak{C}^{(n)}_J$ extends to a chain map
\[
\overline{\iota} : \bigoplus\limits_m\big(\left(\ker \partial_n\right)_m \oplus \left(\ker \partial_n\right)_{m+1}\big)  \rightarrow \bigoplus\limits_J \mathfrak{C}^{(n)}_J
\]
with the property $\mathrm{Im}\overline{\iota} = \mathrm{Im}\iota =  \iota(\ker \partial_n) \simeq \ker \partial_n$ by defining the values of the elements in the complement of $\ker \partial_n$ inductively on the degrees, so that they satisfy the chain map condition.

Denote the $J$-component projection of the resulting chain map by
\begin{equation}\label{iclekn}
\left(\text{Eval}_{{J}}^{(n+1)}\right)_1 := \pi_{{J}} \circ \overline{\iota} \circ \pi_{12} : \mathfrak{C}^{(n+1)} \rightarrow \mathfrak{C}^{(n)}_J,
\end{equation}
where $\pi_{12}$ is the projection chain map to $\bigoplus\limits_m\big(\left(\ker \partial_n\right)_m \oplus \left(\ker \partial_n\right)_{m+1}\big).$
We define the chain map
\begin{equation}\label{ahldlfklsa}
\partial_{n+1} : \mathfrak{C}^{(n+1)} \rightarrow \bigoplus\limits_J \mathfrak{C}^{(n)}_J
\end{equation}
by $\partial_{n+1} := \bigoplus\limits_J \left(\text{Eval}_{{J}}^{(n+1)}\right)_1,$ and the $L_{\infty}[1] $-morphism 
\[
\text{Eval}_{{J}}^{(n+1)} := \left\{\left(\text{Eval}_{{J}}^{(n+1)}\right)_k\right\}_{k \geq 1}
\]
by
\[
\left(\text{Eval}_{{J}}^{(n+1)}\right)_k := \begin{cases}
\pi_{{J}} \circ \overline{\iota} \circ \pi_{12} \text{ in }(\ref{iclekn}) & k = 1, \\
0 & k \geq 2.
\end{cases}
\]

\begin{lem}
We have $\iota \circ \overset{\circ}{h}_1 = \overline{\iota} \circ \pi_{12} \circ \overline{h}_1.$
\end{lem}
\begin{proof}
It follows from
\[
\overline{\iota} \circ \pi_{12} \circ \overline{h}_1(w) = \overline{\iota} \circ \pi_{12} \big(\overset{\circ}{h}_1(w), 0, w\big) = \overline{\iota}\big(\overset{\circ}{h}_1(w), 0\big)= \iota(\overset{\circ}{h}_1(w))
\]
for $w \in C_0.$
\end{proof}

We then obtain the relation
\begin{equation}\label{hjipj}
{h}_{J,1} = \pi_J \circ \bigoplus\limits_J {h}_{J,1} = \pi_J \circ \iota \circ \overset{\circ}{h}_1= \pi_J \circ \overline{\iota} \circ \pi_{12} \circ \overline{h}_1 = \left(\text{Eval}_{{J}}^{(n+1)}\right)_1 \circ \overline{h}_1,
\end{equation}
as we have $\mathrm{Im}\overset{\circ}{h}_1 \subset \ker \partial_n.$ 

To define the chain map
\[
\mathrm{Incl}^{(n+1)} : C \rightarrow \mathfrak{C}^{(n+1)},
\]
we observe that
\[
\mathrm{Im} \left( \bigoplus\limits_J \mathrm{Incl}^{(n)}_{J,1} \right) \in \ker \partial_{n}
\]
holds by the axiom (v) of Definition \ref{hhtp}. As a consequence, we obtain a chain map (with the same notation) $\bigoplus\limits_J \mathrm{Incl}^{(n)}_{J,1} : C \rightarrow \ker \partial_n.$ Composing this with $i_k$ of (\ref{kercy}), we define
\[
\mathrm{Incl}^{(n+1)}_1 := i_k \circ \bigoplus\limits_J \mathrm{Incl}^{(n)}_{J,1}.
\]
for the inclusion $i_k : \ker \partial_{n} \hookrightarrow \mathrm{Cyl}.$
We now proceed with an induction on $k:$ Suppose that we have an $L_{K}[1]$-algebra structure $\left\{l_k : \mathrm{Im} h_{1} \rightarrow \mathfrak{C}^{(n+1)}\right\}_{k \leq K}$ on the subspace $\mathrm{Im}\overline{h}_{1} \subset \mathfrak{C}^{(n+1)},$ and that the family $\left\{\overline{h}_k \in \text{Hom}\left(S^{k}{C}_0, \mathfrak{C}^{(n+1)}\right)\right\}_{k \leq K}$ forms an $L_{K}[1]$-morphism. Then by the fact that $\left(\mathrm{Eval}^{(n+1)}_J\right)_1\bigg|_{\ker \partial_n} : \ker \partial_n \rightarrow \mathfrak{C}_J^{(n)}$ is surjective onto $\ker \partial_n$ contained in $\mathfrak{C}_J^{(n)},$ there exists 
\[
\overline{h}_{K+1} \in \text{Hom}\left(S^{K+1}{C}_0, \ker \partial_n\right),
\]
satisfying
\begin{equation}\label{hevhj}
\left(\mathrm{Eval}^{(n+1)}_J\right)_1 \circ \overline{h}_{K+1} = h_{J,K+1}.
\end{equation}

On the other hand, from the following formula for $\{\overline{h}_k\}_{k \leq K+1}$ being an $L_{K+1}[1]$-morphism
\begin{equation}\label{liftrel}
\begin{split}
l_{K+1} \circ \overline{h}_1^{\otimes K+1} =& \pm l_1 \circ \overline{h}_{K+1} \pm \overline{h}_{K+1} \circ \widehat{l}_1  + \sum\limits_{k_1 + k_2 = K+1} \pm \overline{h}_{k_1} \circ \widehat{l}_{k_2}\\ &+ \sum\limits_{k_1 + \cdots + k_l =K+1} \pm l_{l} \circ (\overline{h}_{k_1}, \cdots, \overline{h}_{k_l}),
\end{split}
\end{equation}
where the signs are determined by the relation (\ref{frel}), we can uniquely determine $l_{K+1} \in \text{Hom}\left(S^{K+1}(\mathrm{Im} h_1), \mathfrak{C}^{(1)}\right)$.

Now we extend $\{l_k\}_{k \leq K+1}$ to an $L_{K+1}[1]$-algebra structure
\begin{equation}\nonumber
\left\{l_k^h \in \text{Hom}\left(S^k{\mathfrak{C}^{(n+1)}}, \mathfrak{C}^{(n+1)}\right) \right\}_{k \leq K+1}
\end{equation}
with the induction hypothesis that $\left\{l_k^h\right\}_{k \leq K}$ is a given $L_K[1]$-algebra structure on $\mathfrak{C}^{(n+1)},$ and that it satisfies
\[
\left(\Eval^{(n+1)}_{J}\right)_1 \circ l^h_k = l_{k} \circ \left(\Eval^{(n+1)}_{J}\right)_1^{\otimes k}, \ 1 \leq k \leq K.
\] 

\begin{lem}
\begin{enumerate}[label = (\roman*)]
\item We have
\[
\left(\Eval^{(n+1)}_{J}\right)_1 \circ l_{K+1} = l_{K+1} \circ \left(\Eval^{(n+1)}_{J}\right)_1^{\otimes K+1}
\]
on $S^{K+1}(\mathrm{Im}\overline{h}_1).$
\item There exists 
\[
\eta_1 \in \mathrm{Hom}\left(S^{K+1}\mathfrak{C}^{(n+1)}, \mathfrak{C}^{(n+1)}\right)
\]
with the following properties:
\begin{enumerate}
\item[(a)]  $\eta_1$ extends $l_{K+1}.$
\item[(b)] For each $J \subset \{0, \cdots, n+1\}$ with $|J| = n+1,$ we have
\begin{equation}\label{elle}
\left(\Eval^{(n+1)}_{J}\right)_1 \circ\eta_1 = l_{K+1} \circ \left(\Eval^{(n+1)}_{J}\right)_1^{\otimes K+1}.
\end{equation}
\end{enumerate}
\end{enumerate}
\end{lem}

\begin{proof}
\noindent (i) Applying $\left(\mathrm{Eval}_J^{(n+1)}\right)_1$ to (\ref{liftrel}), we have
\begin{equation}\label{evfle}
\begin{split}
\left(\mathrm{Eval}_J^{(n+1)}\right)_1 \circ l_{\otimes K+1} &\circ {\overline{h}_{1}}^{K+1}
= 
\pm \left(\mathrm{Eval}_J^{(n+1)}\right)_1 \circ l_1 \circ \overline{h}_{K+1}
\pm \left(\mathrm{Eval}_J^{(n+1)}\right)_1 \circ \overline{h}_{K+1} \circ \widehat{l}_1\\
&\quad \quad \quad \quad \quad + \sum_{k_1 + k_2 = K+2} \pm \left(\mathrm{Eval}_J^{(n+1)}\right)_1 \circ \overline{h}_{k_1} \circ \widehat{l}_{k_2}\\ 
&\quad \quad \quad \quad \quad + \sum_{k_1 + \cdots + k_{\ell} = K+1} \pm \left(\mathrm{Eval}_J^{(n+1)}\right)_1 \circ l_\ell \circ (\overline{h}_{k_1}, \dots, \overline{h}_{k_\ell})\\
&= \pm l_1 \circ \left(\mathrm{Eval}_J^{(n+1)}\right)_1 \circ \overline{h}_{K+1}
\pm {h}_{J,K+1} \circ \widehat{l}_1
+ \sum_{k_1 + k_2 = K+2} \pm h_{J,k_1} \circ \widehat{l}_{k_2}\\
&\quad \quad \quad \quad \quad \quad + \sum_{k_1 + \cdots + k_{\ell} = K+1} \pm l_\ell \circ \left(\mathrm{Eval}_J^{(n+1)}\right)_1^{\otimes l} \circ (\overline{h}_{k_1}, \dots, \overline{h}_{k_\ell})\\
&= \pm l_1 \circ {h}_{J,K+1} \pm {h}_{J,K+1} \circ \widehat{l}_1
+ \sum_{k_1 + k_2 = K+1} \pm {h}_{J, k_1} \circ \widehat{l}_{k_2}\\
&\quad \quad \quad \quad \quad \quad \quad \quad + \sum_{k_1 + \cdots + k_{\ell} = K+1} \pm l_\ell \circ ({h}_{J,k_1}, \dots, {h}_{J,k_\ell}).
\end{split}
\end{equation}
On the other hand, we have
\[
l_{K+1} \circ \left(\mathrm{Eval}_J^{(n+1)}\right)_1^{\otimes {K+1}} \circ \overline{h}_{1}^{\otimes {K+1}} = l_{K+1} \circ ({h}_{J,1}, \dots, {h}_{J,1}),
\]
which equals the last line of (\ref{evfle}) by the fact that $h_J$ is an $L_\infty[1]$-morphism. Thus, we have
\[
\left(\mathrm{Eval}_J^{(n+1)}\right)_1 \circ l_{K+1} \circ {\overline{h}_{1}}^{K+1} = l_{K+1} \circ \left(\mathrm{Eval}_J^{(n+1)}\right)_1^{\otimes {K+1}} \circ \overline{h}_{1}^{\otimes {K+1}},
\]
and in other words, we have
\begin{equation}\label{lkahgalgkj}
\left(\Eval^{(n+1)}_{J}\right)_1 \circ l_{K+1} = l_{K+1} \circ \left(\Eval^{(n+1)}_{J}\right)_1^{\otimes K+1}
\end{equation}
on $S^{K+1}(\mathrm{Im}\overline{h}_1).$

\noindent(ii) Since $S^{K+1}\left(\mathrm{Im}\overline{h}_1\right)$ is a direct summand of $S^{K+1}\mathfrak{C}^{(n+1)},$ from (i), it suffices to show that (b) holds for some $\eta_1$ defined on $\left(S^{K+1}\mathrm{Im}\overline{h}_1\right)^\mathrm{c}.$ Here, $(\cdots)^\mathrm{c}$ denotes the complement.
\begin{equation}\nonumber
\begin{split}
\left(\Eval_{J_1 \cap J_2}^{(n)}\right)_1 \circ l_{K+1} \circ {\left(\Eval_{J_1}^{(n+1)}\right)}_1^{\otimes K+1} & =l_{K+1} \circ {\left(\Eval_{J_1 \cap J_2}^{(n)}\right)}_1^{\otimes K+1} \circ {\left(\Eval_{J_1}^{(n+1)}\right)}_1^{\otimes K+1},\\
\left(\Eval_{J_1 \cap J_2}^{(n)}\right)_1 \circ l_{K+1} \circ {\left(\Eval_{J_2}^{(n+1)}\right)}_1^{\otimes K+1}
&=l_{K+1} \circ {\left(\Eval_{J_1 \cap J_2}^{(n)}\right)}_1^{\otimes K+1} \circ {\left(\Eval_{J_2}^{(n+1)}\right)}_1^{\otimes K+1},
\end{split}
\end{equation}
coincide for each pair $J_1 \neq J_2$ with $|J_1|=|J_2| = n+1$ and $|J_1 \cap J_2| = n$, so that we have
\[
\mathrm{Im}\left(l_{K+1} \circ {\left(\Eval_{J_1}^{(n+1)}\right)}_1^{\otimes K+1} - l_{K+1} \circ {\left(\Eval_{J_2}^{(n+1)}\right)}_1^{\otimes K+1}\right)  \in \ker \left( \mathrm{Eval}^{(n)}_{J_1 \cap J_2}\right)_1
\]
for every $J_1$ and $J_2.$ Then one can show that
\[
\sum\limits_J l_{K+1} \circ {\left(\Eval_{J}^{(n+1)}\right)}_1^{\otimes K+1}\left(\left(S^{K+1}\mathrm{Im}\overline{h}_1\right)^c\right) \in \text{ker} \partial_{n}
\]
with the definition $\partial_n := \bigoplus\limits_{\substack{J \subset \{0, \cdots, n\}, \\ |J| = n}} \left(\Eval_J^{(n)}\right)_1$ in (\ref{ppnnk}) and the relation (\ref{lkahgalgkj}).

Then from the definitions $\partial_{n+1} := \bigoplus\limits_J \left(\text{Eval}_{{J}}^{(n+1)}\right)_1 $ in (\ref{ahldlfklsa}) and $\mathfrak{C}^{(n+1)} := \mathrm{Cyl}$ in (\ref{adsgaahahhh}), and the fact that $\mathfrak{C}^{(n+1)} \supset \ker \partial_n,$ we obtain 
\[
\eta_1 \in \text{Hom}\left(\left(S^{K+1}\mathrm{Im}\overline{h}_1\right)^\mathrm{c}, \mathfrak{C}^{(n+1)}\right),
\]
satisfying
\begin{equation}\label{beta}
\partial_{n+1} \circ \eta_1 = \sum\limits_J l_{K+1} \circ {\left(\Eval_{J}^{(n+1)}\right)}_1^{\otimes K+1}\bigg|_{\left(S^{K+1}\mathrm{Im}\overline{h}_1\right)^\mathrm{c}}.
\end{equation}
Then we obtain
\begin{equation}\label{elle}
\left(\Eval^{(n+1)}_{J}\right)_1 \circ\eta_1 = l_{K+1} \circ \left(\Eval^{(n+1)}_{J}\right)_1^{\otimes K+1}\bigg|_{\left(S^{K+1}\mathrm{Im}\overline{h}_1\right)^\mathrm{c}}.
\end{equation}
by projecting (\ref{beta}) to each component $\mathfrak{C}^{(n)}_{J}.$ 
\end{proof}

Notice that $\{l^h_{\leq K+1}\} := \{l_k^h\}_{k \leq K} \cup \{\eta_1\}$ need not satisfy the $L_{\infty}[1]$-relation (\ref{quadrel}) yet. So, we consider the obstruction class
\begin{equation}\nonumber
O_{K+1}\left(\{l^h_{\leq K+1}\}\right) := \sum\limits_{\substack{k_1 +k_2 = K+2\\ k_1, k_2 \geq 2}} \pm l_{k_1}^h \circ \widehat{l}_{k_2}^h + \delta_1 (\eta_1)
\end{equation}
with the Hochschild differential $\delta_1 := l^h_1 \circ (\cdot) - (\cdot) \circ \widehat{l}^h_1$ on $\text{Hom}\left(S^{K+1}{\mathfrak{C}^{(n+1)}}, \mathfrak{C}^{(n+1)}\right).$

\begin{lem}\label{lem}
$O_{K+1}(\{l^h_{\leq K+1}\})$ satisfies:
\begin{enumerate}[label = (\roman*)]
\item $\delta_1 O_{K+1}\left(\{l^h_{\leq K+1}\}\right) = 0,$
\item $O_{K+1}\left(\{l^h_{\leq K+1}\}\right)|_{\mathrm{Im} \overline{h}_1} = 0,$
\item $\left(\mathrm{Eval}^{(n+1)}_J\right)_1 \circ O_{K+1}\left(\{l^h_{\leq K+1}\}\right) = 0$ for each $J,$
\item $\partial_{n+1 }\circ O_{K+1}\left(\{l^h_{\leq K+1}\}\right) = 0.$
\end{enumerate}
\end{lem}
\begin{proof}
(i) follows from a straightforward computation. For (ii), we obtain
\[
\begin{split}
0 = \sum_{\ell \leq K} &\sum_{\substack{k_1 + \cdots + k_\ell = K+1, \\ m_1 + m _2 = \ell+1}} l_{m_1}^h \circ \widehat{l}_{m_2}^{{h}} \circ (\overline{h}_{k_1}, \cdots, \overline{h}_{k_\ell})\\
&+\sum_{\substack{m_1 + m_2 = K+2, \\ m_1, m_2 \geq 2}} l_{m_1}^h \circ \widehat{l}_{m_2}^{{h}} \circ (\overline{h}_1, \cdots, \overline{h}_1)  + \delta_1({\eta}_1) \circ (\overline{h}_1, \cdots, \overline{h}_1)
\end{split}
\]
by applying $\l_{(\cdots)}$ to the left of (\ref{liftrel}) and taking the sum properly. We also observe that all the terms except the case of $k_1 = \cdots = k_\ell = 1$ must be zero, as $\{l_k^h\}_{k \leq K}$ is an $L_K[1]$-algebra by the induction hypothesis. Then we are left with the terms:
\[
0 = \sum_{\substack{m_1 + m_2 = K+2, \\ m_1, m_2 \geq 2}} l_{m_1}^h \circ \widehat{l}_{m_2}^{{h}} \circ (\overline{h}_1, \cdots, \overline{h}_1) + \delta_1 (\eta_1) \circ (\overline{h}_1, \cdots, \overline{h}_1),
\]
which says that $O_{k+1}\left(\{l^h\}\right)\big|_{\mathrm{Im}\overline{h}_1} = 0$.
(iv) and (v) follow from (\ref{elle}) and (iv), respectively.

\end{proof}
Consider the chain map
\[
\begin{split}
\Xi : \text{Hom}\left(S^{K+1}{\mathfrak{C}^{(n+1)}},\ker \partial_{n+1}\right) &\rightarrow \text{Hom}\left(S^{K+1}({\mathrm{Im}\overline{h}_1}), \ker \partial_{n+1}\right), \\
\xi &\mapsto \xi |_{\mathrm{Im}\overline{h}_1},
\end{split}
\]
where both sides are equipped with the differential $\delta_1.$ Since $\overline{h}_1$ is a quasi-isomorphism. Then we have:

\begin{lem}
$\Xi$ is also a quasi-isomorphism.
\end{lem}

\begin{proof}
Consider the following map in the opposite direction
\begin{equation}\nonumber
\begin{split}
\Psi : \text{Hom}\left(S^{K+1}({\mathrm{Im}\overline{h}_1}), \ker \partial_{n+1}\right) &\rightarrow \text{Hom}\left(S^{K+1}{\mathfrak{C}^{(n+1)}},\ker \partial_{n+1}\right),\\
 \xi \mapsto \overline{\xi},
\end{split}
\end{equation}
where $\overline{\xi}$ is the extension of $\xi$ that sets zero on the complement $(S^{K+1}\mathrm{Im}\overline{h}_1)^\mathrm{c}.$ We observe that $\Xi$ is injective and that it is a left inverse to $\Psi.$ 

Note $\Psi$ is injective. We claim that the quotient complex
\begin{equation}\label{agkjak}
\begin{split}
\text{Hom}&\left(S^{K+1}{\mathfrak{C}^{(n+1)}},\ker \partial_{n+1}\right)/\mathrm{Im} \Psi \simeq \left( \left(S^{K+1}{\mathfrak{C}^{(n+1)}}\right)^{*} \otimes \ker \partial_{n+1} \right) /\mathrm{Im} \Psi\\
& \simeq \left( S^{K+1}\left({\mathfrak{C}^{(n+1)}}\right)^{*} \otimes \ker \partial_{n+1} \right) /\mathrm{Im} \Psi.
\end{split}
\end{equation}
Using the decomposition
\[
{\mathfrak{C}^{(n+1)}} \simeq \frac{\mathfrak{C}^{(n+1)}}{\mathrm{Im}\overline{h}_1} \oplus \mathrm{Im}\overline{h}_1
\]
over a field, we further have
\begin{equation}\label{gagfagdsg}
\begin{split}
(\ref{agkjak})& \simeq \left( S^{K+1}\left(\frac{\mathfrak{C}^{(n+1)}}{\mathrm{Im}\overline{h}_1} \oplus \mathrm{Im}\overline{h}_1
\right)^* \otimes \ker \partial_{n+1}\right)/\mathrm{Im} \Psi\\
& \simeq \left( S^{K+1}\left(\left(\frac{\mathfrak{C}^{(n+1)}}{\mathrm{Im}\overline{h}_1}\right)^* \oplus \left(\mathrm{Im}\overline{h}_1\right)^*
\right) \otimes \ker \partial_{n+1}\right)/\mathrm{Im} \Psi\\
& \simeq \bigoplus\limits_{i > 0} S^i\left(\frac{\mathfrak{C}^{(n+1)}}{\mathrm{Im}\overline{h}_1}\right)^* \otimes \left( \cdots \right) \oplus \left(\left({\mathrm{Im}\overline{h}_1}\right)^* \otimes \ker \partial_{n+1}\right)/\mathrm{Im} \Psi\\
&\simeq \bigoplus\limits_{i > 0} S^i\left(\frac{\mathfrak{C}^{(n+1)}}{\mathrm{Im}\overline{h}_1}\right)^* \otimes \left( \cdots \right) \oplus \mathrm{Hom}\left({\mathrm{Im}\overline{h}_1}, \ker \partial_{n+1}\right)/\mathrm{Im} \Psi\\
&\simeq \bigoplus\limits_{i > 0} S^i\left(\frac{\mathfrak{C}^{(n+1)}}{\mathrm{Im}\overline{h}_1}\right)^* \otimes \left( \cdots \right)
\end{split}
\end{equation}
for some component $(\cdots).$ 

Denote by $(\cdots)^{S_i}$ the set of coinvariants for the symmetric group action of $S_i$. By the standard results on the homology of symmetric products (cf.\ \cite[Subsection 6.1]{Weibel}), we observe that over a field
\[
H_*\left(S^i\left(\frac{\mathfrak{C}^{(n+1)}}{\mathrm{Im}\overline{h}_1}\right)^*\right) \simeq \left(H_*\left(\frac{\mathfrak{C}^{(n+1)}}{\mathrm{Im}\overline{h}_1}\right)^*\right)^{S_i} \simeq 0,
\]
since $\mathfrak{C}^{(n+1)}/\mathrm{Im}\overline{h}_1$ (and hence $\left(\mathfrak{C}^{(n+1)}/\mathrm{Im}\overline{h}_1\right)^*$) is acyclic for quasi-isomorphic $\overline{h}_1$. By applying the K\"{u}nneth formula over a field to (\ref{gagfagdsg}), the desired quasi-isomorphism follows. We conclude that $\Psi$ is a quasi-isomorphism, and hence so is its left inverse $\Xi$.
\end{proof}

By Lemma \ref{lem} (ii) and (iv), $O_{K+1}\left(\{l^h_{\leq K+1}\}\right)$ is contained in the sub-chain complex ker$\Xi,$ which is acyclic, so there exists $\eta_2 \in$ Hom$\left(S^{K+1}{\mathfrak{C}^{(n+1)}}, \ker \partial_{n+1}\right)$ such that $O_{K+1}\left(\{l^h_{\leq K+1}\}\right) = \delta_1 (-\eta_2).$ We denote $l_{K+1}^h : = \eta_1+\eta_2.$ Then one can verify that the family $\{l_{k}^h\}_{k \leq K+1}$ satisfies the $L_{K+1}[1]$-relation. Thus we have constructed an $L_{\infty}[1]$-algebra $\left(\mathfrak{C}^{(n+1)}, \{l^{h}_{k}\}_{k \geq 1}\right)$ and an $L_{\infty}[1]$-morphism $\overline{h} : C_0 \rightarrow \mathfrak{C}^{(n+1)}$ with the property $\left(\mathrm{Eval}^{(n+1)}_J\right)_1 \circ \overline{h} = h_{J}$ from (\ref{hevhj}) and the induction hypothesis. 

\begin{claim}
$\left(\mathfrak{C}^{(n+1)}, \{l^{h}_{k}\}_{k \geq 1}, \left\{\mathrm{Eval}^{(n+1)}_J\right\}_J, \mathrm{Incl}^{(n+1)}\right)$ is a model of $\Delta^{n+1} \times C.$
\end{claim}
\begin{proof}
The axioms (i), (ii), and (iii) in Definition \ref{hhtp} obviously hold.

\noindent (iv) We know $\overline{h}_1 : C_0 \rightarrow \mathfrak{C}^{(n+1)}$ is a quasi-isomorphism, and so is $h_{J,1}$ by the induction hypothesis. Thus, so is $\left(\text{Eval}_{{J}}^{(n+1)}\right)_{1}$ by the relation (\ref{hjipj}).
$\text{Incl}^{(n+1)}$ being a quasi-isomorphism follows from (iv).

\noindent (v) We have
\[
\begin{split}
\left(\text{Eval}_{{J}}^{(n+1)}\right)_1 \circ \text{Incl}_1^{(n+1)} &= \pi_{{J}} \circ \overline{\iota} \circ \pi_{12} \circ i_k \circ \sum\limits_{{J}} \text{Incl}_{{J,1}}^{(n)}\\
& = \pi_{{J}} \circ {\iota} \circ \sum\limits_{{J}} \text{Incl}_{{J,1}}^{(n)} = \text{Incl}_{{J,1}}^{(n)}.
\end{split}
\]

\noindent (vi)
We have
\[
\begin{split}
\ker \partial_n = \overline{\iota} \circ \pi_{12} \left(\mathfrak{C}^{(n+1)}\right) &= \bigoplus\limits_J (\pi_J \circ \overline{\iota} \circ \pi_{12})\left(\mathfrak{C}^{(n+1)}\right)\\ 
&= \bigoplus\limits_J \left(\mathrm{Eval}^{(n+1)}_J\right)_1 \left( \mathfrak{C}^{(n+1)}  \right) = \mathrm{Im}\partial_{n+1}.
\end{split}
\]
\end{proof}
This completes the proof of Proposition \ref{pphhe}.
\end{proof}

\appendix

\section{Proof of Theorem \ref{th}}

This section is devoted to the proof of Theorem \ref{th}. 

We begin with a lemma.
\begin{lem}
There exists a chain map $\mathfrak{F}_1 : \mathfrak{C}_1 \rightarrow \mathfrak{C}_2$ over $f_1$ that is compatible with $\mathrm{Eval}_j, \ j =0,1$ and $\mathrm{Incl},$
\[
\begin{tikzcd}
C_1 \arrow[swap]{d}{f_1} \arrow{r}{\Incl_1} &  \mathfrak{C}_1 \arrow[swap]{d}{\mathfrak{F}_1}\\
C_2 \arrow{r}{\Incl_1} & \mathfrak{C}_2.
\end{tikzcd}
\]
\end{lem}

\begin{proof}
Denote $\mathfrak{F}_1' = \Incl_1 \circ f_1 \circ (\Eval_{s=1})_1$ and consider
\begin{equation}\nonumber
\text{\rm Err1} = (\text{\rm Err1}_0, \text{\rm Err1}_1) \in \text{Hom}(\mathfrak{C}_1, C_2) \oplus \text{Hom}(\mathfrak{C}_1, C_2),
\end{equation}
where we write
\begin{equation}\nonumber
\text{\rm Err1}_j := (\Eval_j)_1 \circ \mathfrak{F}'_1 - f_1 \circ (\Eval_j)_1, \ j=0,1.
\end{equation}
Then we can readily verify that $\delta_1 \text{\rm Err1} =0 $ and $\text{\rm Err1} \circ \Incl_1 = 0.$ Here $\delta_1$ is the coboundary map on $\text{Hom}(\mathfrak{C}_1, C_2) \oplus \text{Hom}(\mathfrak{C}_1, C_2),$ induced by $l_1$ maps on $\mathfrak{C}_1$ and $C_2.$
\end{proof}

We also quote the following general algebraic lemma from \cite{FOOO1}.

\begin{lem}\cite[Lemma 4.4.3]{FOOO1}\label{lem:fooo-lemma}
Let $R$ be a coefficient ring. Consider cochain complexes
$(\overline D_j,d)$, $j=1,\, 2,\, 3$ and
 a cochain homomorphism $i : \overline D_1 \to \overline D_2$
 over $R$. Suppose that $i$ is a cochain homotopy equivalence that
is split injective as an $R$-module homomorphism.
Then for $A \in Hom_R(\overline D_2,\overline D_3)$ such that
$dA = 0$, and $A\circ i = 0,$ there exists $B \in Hom_R(\overline D_2,\overline D_3)$ such
that $dB = A$ and $B\circ i = 0$.
\end{lem}

The following proposition may be regarded as the $L_K[1]$-version of
Theorem \ref{th}.

\begin{prop}\cite[Proposition 4.4.11]{FOOO1}
For an $L_{K-1}[1]$-morphism $\mathfrak{F}^{(K-1)} : \mathfrak{C}_1 \rightarrow \mathfrak{C}_2$ over $f$ that is compatible with $\mathrm{Eval}_j, \ j =0,1$ and $\mathrm{Incl},$ there exists an $L_{K}[1]$-morphism $\mathfrak{F}^{(K)} : \mathfrak{C}_1 \rightarrow \mathfrak{C}_2$ (over $f$) that extends $\mathfrak{F}^{(K-1)}$ and is compatible with $\mathrm{Eval}_j, \ j =0,1$ and $\mathrm{Incl}.$
\end{prop}
\begin{proof}
We denote
\[
\text{\rm Err}_K := \widehat{l} \circ \widehat{\mathfrak{F}}^{(K-1)} - \widehat{\mathfrak{F}}^{(K-1)} \circ \widehat{l}  \in \text{Hom}\left(S^K \mathfrak{C}_1, S^K\mathfrak{C}_2\right),
\]
where $\widehat{(\cdot)}$ denotes the coalgebra map determined by $(\cdot)$ (cf. Lemma \ref{aglem}).

We now list several properties of $\text{\rm Err1}_K$ as a lemma
\begin{lem}\label{lemivp}
We have:
\begin{enumerate}
\item[(i)] $\text{\rm Err}_K|_{S^{\leq K-1}\mathfrak{C}_1} \equiv 0,$  where the notation $S^{\leq K-1}$ is introduced in (\ref{skcn}).
\item[(ii)] $\mathrm{Im}(\text{\rm Err}_K) \subset S^1 \mathfrak{C}_2 = \mathfrak{C}_2.$
\item[(iii)]  $\text{\rm Err}_K \subset \mathrm{Im} \delta_1 \text{ in Hom}\left(S^K \mathfrak{C}_1, S^K\mathfrak{C}_2\right).$
\item[(iv)] There exists $\mathfrak{F}'_K$ such that $\mathfrak{F}'_K \circ \widehat{\Incl} = \Incl_1 \circ f_K.$
\end{enumerate}
\end{lem}
\begin{proof}
(i) and (ii) can be easily verified. For (iii), we consider
\begin{equation}\label{errincl}
\begin{split}
\text{\rm Err}_K \circ \widehat{\Incl} &= \left(l_1 \circ \widehat{\mathfrak{F}}^{(K-1)} - \mathfrak{F}^{(K-1)} \circ \widehat{l}\right) \circ \widehat{\Incl}\\
&= \Incl_1 \circ \left(l_1 \circ \widehat{f}^{(K-1)} - f^{(K-1)} \circ \widehat{l}\right) = - \delta_1\left(\Incl_1 \circ \widehat{f}_K\right).
\end{split}
\end{equation}
Since Incl is a quasi-isomorphism, we have $\text{\rm Err}_K \subset \mathrm{Im} \delta_1.$

For (iv), we consider $\mathfrak{F}_K''$ from (iii) such that $\delta_1 \left(\mathfrak{F}_K''\right) + \text{\rm Err}_K=0.$ Then (\ref{errincl}) implies that we have $\delta_1 \left(\mathfrak{F}_K'' \circ \widehat{\Incl} - \Incl_1 \circ f_K\right) = 0,$ so there exist
$\gamma \in \text{Hom}\left(S^K \mathfrak{C}_1, \mathfrak{C}_2\right)$ with $\delta_1 (\gamma) = 0$ and $\gamma' \in \text{Hom}\left(S^K C_1, \mathfrak{C}_2\right)$ and
\begin{equation}\nonumber
\mathfrak{F}_K'' \circ \widehat{\Incl} - \Incl_1 \circ f_K =  \delta_1({\gamma}
 \circ \widehat{\Incl} + \gamma') = {\gamma}
 \circ \widehat{\Incl} + \delta_1(\gamma').
\end{equation}
Then $\mathfrak{F}_K' := \mathfrak{F}_K'' - \gamma - \delta_1\left(\gamma' \circ \widehat{\Eval}\right)$ satisfies the desired condition.
\end{proof}

 Since $\mathfrak{F}^{(K-1)}$ is an $L_{K-1}[1]$-morphism,
$\mathfrak{F}^{(K-1)}$ together with $\mathfrak{F}_K'$
defines an $L_K[1]$-morphism by the property (iv)
of $\mathfrak{F}_K'$ in Lemma \ref{lemivp}. (See \cite[Lemma 4.4.18]{FOOO1} for the $A_K$-case.) Let $\mathfrak{F}^{(K)\prime}$ be the resulting
$L_{K}[1]$-morphism obtained thereby. $\mathfrak{F}^{(K)\prime}$, however, may not yet be compatible with
the evaluation map $\Eval_j$'s, and so we need to modify it.

We denote by
\begin{equation}\nonumber
\text{\rm Err}_K^{(j)} := \Eval_{j} \circ \mathfrak{F}^{(K)\prime} - f \circ \widehat{\Eval}_j, \ j =0,1
\end{equation}
the measure of the aforementioned incompatibilities.
Notice that we have
$$
\text{\rm Err}_K^{(j)}|_{S^{\leq K-1}\mathfrak{C}_1}
\equiv 0; \
\mathrm{Im}\left(\text{\rm Err}_K^{(j)}|_{S^{\leq K}\mathfrak{C}_1}\right)
\subset C_2, \ j =0,1
$$
hence we have
\begin{equation}\nonumber
\left(\Err_K^{(0)}, \Err_K^{(1)}\right) \in \text{Hom}\left(S^{\leq K} \mathfrak{C}_1, C_2 \oplus C_2\right).
\end{equation}
We can straightforwardly verify that $\delta_1 \left(\Err_K^{(0)}, \Err_K^{(1)}\right) = 0$
(by the fact that $\mathfrak{F}^{(K)\prime}$ and $f$ are
$L_{K}[1]$-morphisms) and that
$\text{\rm Err}_K^{(j)} \circ \widehat{\Incl} = 0$ (by the assumption of compatibility with $\Incl$).

In fact, Lemma \ref{lem:fooo-lemma} states that there exists 
\begin{equation}\nonumber
\left(\text{\rm Cor1}^{(0)}_K, \text{\rm Cor1}^{(1)}_K\right) \in \text{Hom}\left(S^{\leq K} \mathfrak{C}_1, C_2 \oplus C_2\right),
\end{equation}
satisfying
\beastar
\text{\rm Cor1}_K^{(j)} \circ \widehat{\Incl} & = & 0, \ j=0, 1,\\
\delta_1 \left(\text{\rm Cor1}^{(0)}_K, \text{\rm Cor1}^{(1)}_K\right)
& = & \left(\Err_K^{(0)}, \Err_K^{(1)}\right).
\eeastar
Then from (iii) Definition \ref{mdl1}, $\Eval_j, \ j =0,1,$ we have $\text{Cor2}_K \in \text{Hom}\left(S^K \mathfrak{C}_1, \mathfrak{C}_2\right)$
such that
\[
\left(\Eval_j\right)_1 \circ \text{Cor2}_K = \text{\rm Cor1}^{(j)}_{K}, \ j =0, 1.
\]
It immediately follows that
\[
\text{Cor2}_K \circ \widehat{\Incl} = 0.
\]
Now we can verify that $\mathfrak{F}_K : = \mathfrak{F}'_K - \text{\rm Cor2}_K$ is the $K$-th multilinear map of the desired $L_{K}[1]$-morphism. The proof of Theorem \ref{th} is now complete.
\end{proof}

\end{document}